\documentclass[a4paper,reqno]{amsart}
\usepackage{csquotes}
\usepackage[british]{babel}

\numberwithin{equation}{section}

\usepackage{amssymb}
\usepackage[protrusion=true,expansion=true]{microtype}	
\usepackage{amsmath,amsfonts,amsthm} 
\usepackage[pdftex]{graphicx}	
\usepackage{url}
\usepackage[title,titletoc,page]{appendix} 
\usepackage{listings} 
\usepackage{color}
\usepackage{comment}
\usepackage{caption}
\usepackage{subcaption}
\usepackage{algorithm}
\usepackage{graphicx}
\usepackage{subcaption}
\usepackage{orcidlink}

\usepackage{fancyhdr}
\numberwithin{equation}{section}		
\numberwithin{figure}{section}			
\numberwithin{table}{section}				

\newcommand{\R}{\mathbb{R}}

\usepackage[left=3.5cm,right=3.5cm,top=3.5cm,bottom=3.5cm,footskip=1.5cm,headsep=1cm]{geometry}

\usepackage[T1]{fontenc} 
\usepackage{lmodern} 
\rmfamily 
\DeclareFontShape{T1}{lmr}{b}{sc}{<->ssub*cmr/bx/sc}{}
\DeclareFontShape{T1}{lmr}{bx}{sc}{<->ssub*cmr/bx/sc}{}
\usepackage{lipsum}
\usepackage{mathtools}
\usepackage{amsmath}
\usepackage{amssymb}
\usepackage{tikz}
\usetikzlibrary{matrix,positioning,decorations.pathreplacing,calc}
\usetikzlibrary{
decorations.text,%
decorations.markings,%
shadows}
\usepackage{adjustbox}
\usepackage{graphicx}

\usepackage{dsfont} 
\usepackage[format=hang]{caption}
\usepackage{subcaption}
\usepackage[normalem]{ulem}

\usepackage{bm}

\usepackage[
style=numeric,
sorting=nty,
maxnames=99,
maxalphanames=5,
natbib=true,
sortcites]{biblatex}

\DeclareNameAlias{default}{family-given}

\graphicspath{{figures/}}

\AtEveryBibitem{
 \clearfield{url}
 \clearfield{issn}
 \clearfield{isbn}
 \clearfield{urldate}
 
 \ifentrytype{book}{
  \clearfield{pages}}{%
 }
}

\usepackage{xargs}                      
\usepackage[prependcaption,textsize=tiny,textwidth=3cm]{todonotes}
\newcommandx{\unsure}[2][1=]{\todo[linecolor=red,backgroundcolor=red!25,bordercolor=red,#1]{#2}}
\newcommandx{\change}[2][1=]{\todo[linecolor=blue,backgroundcolor=blue!25,bordercolor=blue,#1]{#2}}
\newcommandx{\info}[2][1=]{\todo[linecolor=OliveGreen,backgroundcolor=OliveGreen!25,bordercolor=OliveGreen,#1]{#2}}
\newcommandx{\improvement}[2][1=]{\todo[linecolor=black,backgroundcolor=black!25,bordercolor=black,#1]{#2}}
\newcommandx{\thiswillnotshow}[2][1=]{\todo[disable,#1]{#2}}

\usepackage{amsthm}
\usepackage[noabbrev,capitalize]{cleveref}
\crefname{proposition}{Proposition}{Propositions}
\crefname{equation}{}{}

\newtheorem{theorem}{Theorem}[section]
\newtheorem{lemma}[theorem]{Lemma}
\newtheorem{proposition}[theorem]{Proposition}

\theoremstyle{definition}
\newtheorem{definition}[theorem]{Definition}

\newtheorem{assumption}[theorem]{Assumption}
\newtheorem{remark}[theorem]{Remark}

\crefname{assumption}{Assumption}{Assumptions}
\crefname{definition}{Definition}{Definitions}
\crefname{corollary}{Corollary}{Corollaries}
\crefname{enumi}{item}{items}

\usepackage{fancyhdr}

\fancypagestyle{plain}{%
\fancyhead[C]{}
\fancyfoot[C]{}

}
\fancypagestyle{nosection}{%
\fancyhead[CE]{}
\fancyhead[CO]{}
\fancyfoot[CE,CO]{\thepage}
}

\DeclareMathOperator{\Z}{\mathbb{Z}}

\DeclareMathOperator{\C}{\mathbb{C}}

\renewcommand{\tilde}{\widetilde}
\renewcommand{\hat}{\widehat}
\renewcommand{\bar}[1]{\overline{#1}}

\renewcommand{\qedsymbol}{$\blacksquare$}

\usepackage{fancyhdr}

\fancypagestyle{plain}{%
\fancyhead[C]{}
\fancyfoot[C]{}

}
\fancypagestyle{nosection}{%
\fancyhead[CE]{}
\fancyhead[CO]{}
\fancyfoot[CE,CO]{\thepage}
}

\renewcommand{\epsilon}{\varepsilon}

\renewcommand{\tilde}{\widetilde}
\renewcommand{\hat}{\widehat}

\usepackage{tcolorbox}
\usetikzlibrary{tikzmark,calc}

\usepackage{enumerate}
\usepackage{upgreek}

\def\e{\eta}

\def\ue{u_\e}

\def\uek{u_{\e_k}}

\def\tol2{\buildrel\hbox{$L^2$}\over\longrightarrow}

\def\IOm{\int_{\Omega}}

\def\IOAY{\int\!\!\int_{\Omega\times Y^m}}

\def\limf{\displaystyle{\liminf_{\eta\to 0}}}

\def\xxe{\left(\vec{x},\frac{\textbf{P}\vec{x}}{\e}\right)}

\def\xy{(\vec{x},\vec{y})}

\def\wdto{\buildrel\hbox{$\textbf{P}$}\over
\rightharpoonup \!\!\!\!\! \rightharpoonup}

\def\dto{\buildrel\hbox{$\textbf{P}$}\over
\to \!\!\!\!\! \to}

\def\uo{u_0(\vec{x},\vec{y})}

\def\ve{v_\e}

\def\O{\Omega}

\begin{document}
\title[High-frequency spectral asymptotics and homogenization for quasiperiodic operators]{High-frequency spectral asymptotics and homogenization for quasiperiodic operators}

\author[E. Cherkaev]{Elena Cherkaev \,\orcidlink{0000-0002-5682-5453}}
\address{\parbox{\linewidth}{Elena Cherkaev\\
 University of Utah, Department of Mathematics, 
 155 South 1400 East, Salt Lake City, UT 84112, United States of America,  \href{http://orcid.org/0000-0002-5682-5453}{orcid.org/0000-0002-5682-5453}}.}
\email{elena@math.utah.edu}
\thanks{}

\author[B. Davies]{Bryn Davies \,\orcidlink{0000-0001-8108-2316}}
\address{\parbox{\linewidth}{Bryn Davies\\
 University of Warwick, Mathematics Institute, 
 Zeeman Building, Coventry CV4 7AL, United Kingdom, 
 \href{http://orcid.org/0000-0001-8108-2316}{orcid.org/0000-0001-8108-2316}}.}
\email{bryn.davies@warwick.ac.uk}
\thanks{}

\author[S. Guenneau]{S\'ebastien Guenneau \,\orcidlink{0000-0002-5924-622X}}
\address{\parbox{\linewidth}{S\'ebastien Guenneau\\
 Imperial College London, Department of Physics, UMI 2004 Abraham de Moivre CNRS-Imperial,
 Prince Consort Road, London SW7 2AZ, United Kingdom, 
 \href{http://orcid.org/0000-0002-5924-622X}{orcid.org/0000-0002-5924-622X}}.}
\email{s.guenneau@imperial.ac.uk}
\thanks{}

\author[C. Thalhammer]{Clemens Thalhammer \,\orcidlink{0009-0006-7218-260X}}
\address{\parbox{\linewidth}{Clemens Thalhammer\\
 ETH Z\"urich, Department of Mathematics, 
 R\"amistrasse 101, 8092 Z\"urich, Switzerland, 
 \href{https://orcid.org/0009-0006-7218-260X}{orcid.org/0009-0006-7218-260X}}.}
\email{clemens.thalhammer@sam.math.ethz.ch}
\thanks{}

\author[N. Wellander]{Niklas Wellander \,\orcidlink{0000-0002-4211-8977}}
\address{\parbox{\linewidth}{Niklas Wellander\\
 Lule\r{a} University of Technology, Department of Engineering Sciences and Mathematics, 
 SE-971 87 Lule\r{a}, Sweden and  Swedish Defence Research Agency, Department of Robust Radar Systems, Linköping, Sweden, \href{https://orcid.org/0000-0002-4211-8977}{orcid.org/0000-0002-4211-8977}}.}
\email{niklas.wellander@ltu.se}
\thanks{}

\begin{abstract}
We study the spectral asymptotics of elliptic operators with quasiperiodic coefficients by exploiting projections from higher-dimensional periodic functions. Using the framework of two-scale convergence adapted to cut-and-project quasiperiodic structures we establish that, in the low-frequency (homogenization) regime, the spectrum converges to that of a homogenized operator with effective coefficients determined by a cell problem on the higher-dimensional torus. In the high-frequency regime, we introduce a rescaling approach that transforms the problem to an expanding domain with asymptotically frozen coefficients. In the critical scaling, the rescaled spectrum converges to the union of the Bloch spectra arising from the quasiperiodic bulk and a boundary layer spectrum consisting of eigenfunctions concentrated near the boundary of the macroscopic domain. This boundary spectrum can be characterised as a subset of the spectrum of a family of half-space operators with frozen macroscopic coefficients. For any non-critical scaling, the rescaled spectrum fills the positive real line.
\end{abstract}
\maketitle

\noindent\textbf{Keywords:} Two-scale homogenization, cut-and-project crystals, spectral analysis, Bloch waves, pseudo-eigenfunctions

\section{Introduction}

Quasicrystals are one of the most fascinating phases of matter to have emerged in the last half a century. They were discovered through the observation of unexpected symmetries in diffraction patterns (in particular, ten-fold symmetries \cite{shechtman1984metallic}), and have inspired longstanding and challenging mathematical problems. For example, the spectra of operators with quasiperiodic coefficients are renowned for having exotic and unusual spectral properties such as fractal and Cantor spectra \cite{Simon1982, avila2009ten, nonnenmacher2005fractal}, metal-insulator transitions \cite{jitomirskaya1999metal, morison2022order}, unexpected symmetries in reciprocal space \cite{shechtman1984metallic}, and unique topological properties \cite{liu2022topological, putley2024mixing, agazzi2014colored}. These properties can be challenging to characterise. Notable challenges have included understanding the nature of the spectrum (whether it is pure point, singular continuous or absolutely continuous \cite{avila2009ten}), whether or not the Lebesgue measure of the spectrum is non-zero \cite{sutHo1989singular} and estimating its fractal dimension \cite{damanik2008hyperbolicity}.

Fortunately, in spite of the exotic spectral properties of quasicrystals, there are a variety of standard tools and approaches that can be applied, once modified appropriately. For example, Floquet-Bloch theory can be used to approximate the spectrum if we approximate the aperiodic geometry by a periodic pattern. This is typically achieved either by taking a large \textit{supercell} \cite{davies2025convergence, davies2024super, shubin1978almost, Damanik2016, Avron1990,guenneau2008acoustic} or by lifting the problem into a higher-dimensional periodic \textit{superspace} \cite{rodriguez2008computation, davies2025convergence}. The existence of a higher-dimensional periodic superspace has also facilitated the development of homogenisation methods, which exploit averages taken in the superspace \cite{Bouchitte+etal2010,wellander2018two}. In this work, we add a valuable new tool to this repertoire by showing that the spectrum has asymptotic scaling properties that can be fully characterised through the union of Bloch spectra of operators with frozen quasiperiodic coefficients and boundary layer spectra arising from the interaction with the domain boundary.

Quasicrystals' exotic wave-structure interactions have led to their usage in a variety of applications. Promising properties include the ability to achieve isotropy in lattice structures \cite{chen2020isotropic, wang2020quasiperiodic}, exhibit negative refraction \cite{morini2019negative}, yield zero Poisson's ratio \cite{clarke2023isotropic} and support spectra with non-trivial topologies \cite{kraus2012topological, liu2022topological}. These properties, as well as the ability to unlock a wider range of spectral patterns, have far-reaching implications for applications. Recent examples have included building ``topological'' wave guides \cite{kraus2012topological, Davies2022, beli2023interface} and ``rainbow'' harvesting devices \cite{davies2023graded}.

Our main result, Theorem~\ref{thm: SpectralLimit}, establishes that in the critical high-frequency scaling the rescaled spectrum $\lim_{\eta\to 0}\eta^{2}\sigma_\eta$ can be decomposed as the union of a \emph{Bloch spectrum} $\sigma_{\text{Bloch}}$, obtained by freezing the macroscopic variable and taking the union of spectra of the resulting quasiperiodic operators over $\bar{\Omega}$, and a \emph{boundary layer spectrum} $\sigma_{\text{Boundary}}$, arising from eigenfunctions that concentrate near the boundary $\partial\Omega$. For domains with $\mathcal{C}^{2}$ boundary, we further show that
\begin{equation*}
    \sigma_{\text{Boundary}}\subset\bigcup_{(\vec{x}_0,\vec{y}_0)\in B}\sigma(T_{\vec{x}_0,\vec{y}_0}),
\end{equation*}
where $T_{\vec{x}_0,\vec{y}_0}$ is the half-space operator with coefficients frozen at $\vec{x}_0$ and shifted by $\vec{y}_0$ in the second variable, showing that the boundary spectrum is governed by the local geometry and microstructure at each boundary point. We emphasise that the limiting Bloch spectrum retains the rich structure typical of quasiperiodic operators: in particular, the spectra of the frozen-coefficient operators $T_{\vec{x}_0}$ are known to exhibit spectral gaps and, in many cases, Cantor-like structure \cite{Simon1982, avila2009ten}. The limiting spectrum is therefore expected to be far from trivial and, in contrast to the sub-critical and super-critical regimes where the rescaled spectrum fills $\mathbb{R}_+$, the critical scaling can capture interesting non-trivial spectral features arising from the quasiperiodic microstructure.

The remainder of this paper is organised as follows: In the rest of this section, we recall the cut-and-project framework for quasiperiodic structures and the associated notion of two-scale convergence. In Section~2, we establish the low-frequency homogenisation limit, showing that the spectrum of the quasiperiodic operator converges to that of a homogenised operator. Section~3 is devoted to the high-frequency limit, where we introduce a rescaling approach and treat the critical, sub-critical, and super-critical scalings separately. Section~4 studies the boundary layer spectrum for domains with $\mathcal{C}^{2}$ boundary, showing that it is contained in the union of spectra of half-space operators $T_{\vec{x}_0,\vec{y}_0}$ with frozen macroscopic variable, taken over all boundary points $(\vec{x}_0,\vec{y}_0) \in B$. Finally, we conclude with a summary and outlook in Section~5.

\subsection{Notation}
We will use the following notation conventions throughout this article.

\begin{itemize}
    \item $\lVert \cdot \rVert$ denotes the $L^2(\mathbb{R}^n)$ norm wherever it is clear from context.
    \item $\langle \cdot,\cdot\rangle$ denotes the $L^2$ scalar product.
    \item $\langle \cdot,\cdot\rangle_{H^{-1},H^1}$ denotes the $H^{-1},H^1$ pairing with the convention of sesqui-linearity in the second variable.
    \item Any function space with the subscript $\#$ is a space of periodic functions.
    \item ${\mathcal L}(\Omega)$ is used to denote ${\mathcal L}(L^2(\Omega))$, the space of linear operators $L^2(\Omega)\to L^2(\Omega)$.
    \item $D(X,Y)$ denotes the space of functions from $X$ to $Y$ which are smooth functions whenever $X$ is bounded and Schwartz functions whenever $X$ is unbounded.
    \item $[A]_{\mathcal{C}^{0,\alpha}}$ denotes the $\alpha$-Hölder seminorm.
\end{itemize}
\subsection{Quasi-crystals via cut-and-project method}

Quasiperiodic patterns are those that can be related to higher-dimensional periodic structures. This is achieved by taking a cut along a hyperplane in the higher-dimensional space and projecting onto the lower-dimensional physical space, which is typically $\R^3$. This was proposed by the physicists Duneau and Katz forty years ago \cite{Duneau+Katz1985}. For example, the naturally occurring quasiperiodic mineral,  
$\rm{Al}_{63}\rm{Cu}_{24}\rm{Fe}_{13}$ (icosahedrite),  reported in  \cite{bindi2009}, can be described by a linear mapping such that, e.g. the conductivity is given by
  $\textbf{P}: \R^3 \to \R^6$,
\begin{equation*}
\begin{array}{ll}
A \bigl
( \textbf{P}\vec{x} \bigr )
= &
A \bigl ( n_\tau(x_1 + \tau x_2), n_\tau(\tau x_1 + x_3), n_\tau(x_2 +
\tau x_3), \cr & n_\tau(-x_1 + \tau x_2), n_\tau(\tau x_1 - x_3),
n_\tau(-x_2 + \tau x_3) \bigr ), \cr
\end{array}
\end{equation*}
where $n_{\tau}$ is the normalization constant
$1/\sqrt{2(2+\tau)}$ with the Golden
number  $\tau=(1+\sqrt{5})/2$
and $A \in L^{\infty}_{\sharp}(Y^6; \mathcal{C}_\sharp(Y_\perp^3))$, i.e. the conductivity is bounded almost everywhere on the hypercube $Y^6={]0,1[}^6$, it is continuous on the subset $Y_\perp^3$ orthogonal to the hyperplane of the cut
and is periodic. We note that there is an
ambiguity in the definition of the conductivity as we
could have defined it via a cut-and-projection
from a periodic array in $\R^{12}$ onto $\R^3$ $A(\textbf{P}' \vec{x})$ where $\textbf{P}': \R^3 \to \R^{12}$, i.e. $\textbf{P}'$ is a matrix with $12$ rows and $3$ columns and $A\in L^{\infty}_{\sharp}(Y^{12};\mathcal{C}_\sharp(Y_\perp^3))$ \cite{Duneau+Katz1985}. A simple example with $\textbf{P}:\R\to\R^2$ is shown in Figures~\ref{fig:cutandproject1} and~\ref{fig:cutandproject2}.

The existence of an underlying (higher-dimensional) periodic structure, suggests the possibility to adapt standard periodic homogenization tools to quasiperiodic materials. In particular, we will demonstrate the convergence of two-scale homogenization approach, based on taking a two-scale limit in the higher-dimensional periodic space, as depicted in Figures~\ref{fig:cutandproject1} and~\ref{fig:cutandproject2}.

We now state the standing assumptions on the coefficient matrix that will be imposed throughout this study.
\begin{assumption}\label{ass:A}
Let $\Omega\subset\R^n$ be a bounded domain with Lipschitz boundary and let $Y^m = {]0,1[}^m$. The coefficient matrix $A\colon \bar{\Omega}\times\R^m\to\R^{n\times n}$ satisfies the following:
\begin{enumerate}[(i)]
    \item \emph{Hermiticity.} $A(\vec{x},\vec{y}) = \overline{A}(\vec{x},\vec{y})^T$ for all $(\vec{x},\vec{y})\in\bar{\Omega}\times\R^m$.
    \item \emph{Periodicity.} For each $\vec{x}\in\bar{\Omega}$, the map $\vec{y}\mapsto A(\vec{x},\vec{y})$ is $Y^m$-periodic.
    \item \emph{Uniform ellipticity.} There exists $\gamma>0$ such that
    \begin{equation*}
        A_{ij}(\vec{x},\vec{y})\xi_i\xi_j \geq \gamma\lvert\vec{\xi}\rvert^2 \qquad \text{for all } \vec{\xi}\in\R^n,\ (\vec{x},\vec{y})\in\bar{\Omega}\times\R^m.
    \end{equation*}
    \item \emph{Regularity.} $A\in \mathcal{C}^{0,\alpha}(\bar{\Omega}; \mathcal{C}^{0,\alpha}_\sharp(Y^m;\R^{n\times n}))$. That is, $A$ is $\alpha-$Hölder continuous on $\overline{\Omega}\times Y^m$.
    \item Throughout the manuscript we make use of summation over repeated indices.
\end{enumerate}
\end{assumption}

\begin{remark}
    The regularity in Assumption~\ref{ass:A}(iv) is chosen to provide a unified framework for all results in this paper. For the low-frequency homogenisation results of Section~2, the weaker condition $A\in L^\infty(\Omega; L^\infty_\sharp(Y^m))$ together with (i)--(iii) suffices; the proofs rely only on uniform ellipticity, boundedness, and the two-scale convergence framework. The full $\mathcal{C}^{0,\alpha}$ regularity in both variables is needed in Section~3 for the high-frequency analysis, where bounds on $[A]_{\mathcal{C}^{0,\alpha}}$ enter the error estimates for the frozen-coefficient approximation. For the boundary layer analysis in Section~4, $\alpha-$Hölder continuity of $A$ in $\vec{x}$ (uniformly in $\vec{y}$) is the key requirement.
\end{remark}
\subsection{Cut-and-project two-scale convergence}
We consider a real-valued rectangular matrix $\textbf{P}$ with $m$ rows and $n$ columns ($n<m$), which fulfils the
criterion
\begin{equation}
\textbf{P}^T \vec{k} \not = {\bf 0} \; , \; \forall \vec{k} \in
\Z^m \setminus \{{\bf 0}\},   \label{criterion}
\end{equation}
so that it has some incommensurate entries. This is the criterion that is needed for the projected function to be non-periodic. Under this assumption, our goal will be to approximate the oscillating sequence of functions  $(u_\e(\vec{x}))_{\eta\in ]0,1[}$ (denoted $(u_\e)$ in the sequel), by a
sequence of two-scale locally quasiperiodic   $u_0\left({\vec{x}},\frac{\textbf{P}\vec{x}}{\e}\right)$ where $u_0\left({\vec{x}},\cdot\right)$ is $Y^m$-periodic on $\R^m$.

\begin{figure}
    \centering
   \includegraphics[width=0.45\linewidth]{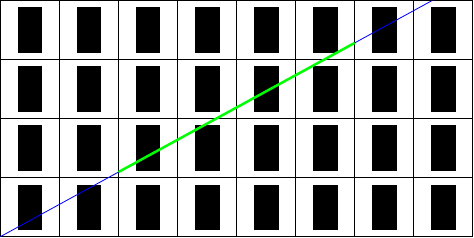}
   \hfill
    \includegraphics[width=0.45\linewidth]{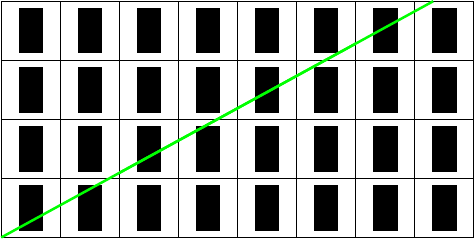}
    \caption{A cut-and-project quasiperiodic structure with the one-dimensional domain size (indicated by the thick green line) increased while the two-dimensional fundamental cell size is fixed.}
    \label{fig:cutandproject1}
\end{figure}

\begin{figure}
    \centering
   \includegraphics[width=0.45\linewidth]{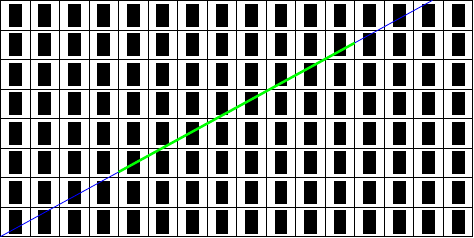}
   \hfill
    \includegraphics[width=0.45\linewidth]{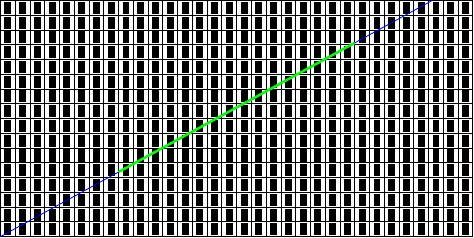}
    \caption{The same cut-and-project construction as Figure~\ref{fig:cutandproject1} but with the domain size fixed and the fundamental cell size decreased.}
    \label{fig:cutandproject2}
\end{figure}

We adopt the following definition of two-scale convergence
associated with a matrix $\textbf{P}$.
\begin{definition}[Weak \textbf{P}-two-scale convergence]
Let $\Omega$ be an open set in $\R^n$ \footnote{Note that throughout Section 1.2, $\Omega$ need not be bounded and can be the whole of $\R^n$.}and $Y^m={]0,1[}^m$.
We say that the sequence $(u_\e)$ two-scale converges weakly towards the function $u_0\in L^2(\O\times Y^m)$ for a matrix $\textbf{P}$  if for every $\varphi\in L^2(\O,\mathcal{C}_\sharp(Y^m))$
  \begin{equation}\label{eq:weak2scale}
\lim _{\e\to 0} \IOm\ue (\vec{x}) \varphi\left(\vec{x}, \frac{\textbf{P}\vec{x}}{\e}\right) \,\mathrm{d}\vec{x} = \IOAY \uo \textcolor{black}{\varphi}(\vec{x},\vec{y})\,\mathrm{d}\vec{x}\mathrm{d}\vec{y}.
  \end{equation}
\end{definition}
We denote weak two-scale convergence for a matrix $\textbf{P}$ with   $\ue \wdto u_0$. 
The following result \cite[Proposition~2.3]{Bouchitte+etal2010} ensures the existence of such two-scale limits when
the sequence $(u_\e)$ is bounded in $L^2(\O)$ and $\textbf{P}$
satisfies (\ref{criterion}).

\begin{proposition}\label{prop:weaktwoscaleprop}
Let $\O$ be an open set in $\R^n$ and $Y^m={]0,1[}^m$. If
$\textbf{P}: \R^n \to \R^m$ is a linear map   satisfying (\ref{criterion}) and $(\ue)$ is a
bounded sequence in $L^2(\O)$, then there exist a vanishing
subsequence $\e_k$ and a limit $\uo \in L^2(\O\times Y^m)$
($Y^m$-periodic in $\vec{y}$) such that $\uek \wdto u_0$ as $\e_k \to 0$.
\end{proposition}
\noindent
We will need to pass to the limit in integrals
$\displaystyle{\int_{\Omega} u_\e \; v_\e \, \mathrm{d}\vec{x}}$ where
$u_\e\wdto u_0$ and $v_\e\wdto v_0$. For this, we introduce 
the notion of strong two-scale
(cut-and-projection) convergence for a matrix $\textbf{P}$.
\begin{definition}[Strong \textbf{P}-two-scale convergence]
A sequence $u_\e$ in $L^2(\O)$ is said to two-scale converge strongly, for a matrix $\textbf{P}$, towards a limit $u_0$ in
$L^2_\sharp (\O\times Y^m)$, which we denote $u_\e\dto u_0$, if and
only if $\ue \wdto u_0$   and
\begin{equation}
\displaystyle{ {\Vert u_\e(\vec{x})\Vert}_{L^2(\O)} \to{\Vert u_0
(\vec{x},\vec{y}) \Vert}_{L^2(\O\times Y^m)}},
\end{equation}
as $\e \to 0$.   
\end{definition}

This definition expresses that the effective oscillations of the
sequence $(u_\e)$ have a quasi-periodicity that is on the order of $\e$. Moreover,
these oscillations are fully identified by $u_0$.
The following proposition \cite[Proposition~2.10]{Bouchitte+etal2010} provides us with
a corrector type result for the sequence
$u_\e$ when its limit $u_0$ is smooth enough.
\begin{proposition}
Let $\textbf{P}$ be a linear map from $\R^n$ to ${\R^m}$ satisfying
(\ref{criterion}). Let $\ue$ be a sequence bounded in $L^2(\O)$ such
that $\ue\wdto u_0$ (weakly).   Then

i) $\ue$ weakly converges in $L^2(\O)$ towards $\displaystyle{u(\vec{x}) = \int_{Y^m} u_0(\vec{x},\vec{y}) \; \mathrm{d}\vec{y}}$ and
\begin{equation}
\limf {\Vert\ue\Vert}_{L^2(\O)}\geq {\Vert u_0 \Vert}_{L^2(\O \times
Y^m)} \geq {\Vert u\Vert}_{L^2(\O)}.
\end{equation}

ii) Let $\ve$ be another bounded sequence in $L^2(\O)$ such that
$v_\e\dto v_0$ (strongly), then
\begin{equation}
\ue\ve\to  w(\vec{x}) \; \hbox{ in ${\mathcal D}'(\O)$ where
$\displaystyle{w(\vec{x}) = \int_{Y^m}\uo v_0(\vec{x},\vec{y})
\,\mathrm{d}\vec{y}}$}.
\end{equation}

iii) Let us further assume that $u_0\in L^2 (\O, \mathcal{C}_\sharp (Y^m))$; if ${\displaystyle\ue\dto  u_0}$ strongly and
$$\displaystyle {{\left\Vert u_0\xxe \right\Vert}_{L^2(\O)} \to{\Vert
u_0\Vert}_{L^2(\O\times Y^m)}}$$
then
\begin{equation}
\displaystyle{{\left\| \ue - u_0\xxe \right\|}_{L^2(\O)}\to  0}.
\end{equation}
\end{proposition}

Classes of functions such that $\displaystyle {{\left\| u_0 \xxe
\right\|}_{L^2(\O)} \to{\left\Vert u_0\right\Vert}_{L^2(\O\times Y^m)}}$ are
said to be admissible for the two-scale (cut-and-projection)
convergence. Examples of admissible functions include $L^2(\O,\mathcal{C}_\sharp(Y^m))$, which is a dense subset of $L^2(\O\times Y^m)$ or $C_c(\Omega)\times L^2_\sharp(Y^m)$ (for a discussion of criteria of admissibility in the periodic case, see \cite{Allaire1992}).

When performing homogenisation on a PDE system, the objective is, given that $u_\e\wdto u_0$ and $\nabla u_\e\wdto \chi$, to identify the differential relationship between
$\chi$ and $u_0$, given a bounded sequence $(u_\e)$ in $H^{1}(\O)$.
A classical solution to this problem was presented by Allaire in the case of periodic functions \cite{Allaire1992} and Bouchitt\'e et al. for quasi-periodic functions
\cite{Bouchitte+etal2010}. In the latter case, the oscillations of the sequence
$(\nabla u_\e)$ cannot be represented in general as the usual gradient  of a periodic function.

It has been established in \cite[Proposition~2.13]{Bouchitte+etal2010} that the two-scale
cut-and-projection limit of the gradients of bounded sequences in $H^{1}(\O)$ exists and can be characterised:
\begin{proposition}\label{prop:gradtwoscale}
Let $\textbf{P}$ be a matrix satisfying (\ref{criterion}) and $(\ue)$ a
bounded sequence in $H^{1}(\O)$. Then, there exist $u_0 \in H^{1}(\O)$ and $ \vec{w} \in L^2(\O,{\mathcal L}_\textbf{P})$, and a subsequence (still denoted by $(\ue)$) such that 
\begin{equation*}
\ue\dto u_0(\vec{x}) \;  \mbox{ and } \; \nabla\ue\wdto\nabla u_0(\vec{x}) +  \vec{w} \xy,
\end{equation*}
where
\begin{equation*}
{\mathcal L}_\textbf{P} := \Bigl \{
 \vec{v} \in L^2_{\sharp}(Y^m; \R^n) \,\mid\,\  \vec{v}(\vec{y}) = \sum_{\vec{k}\in \Z^m \setminus \{\vec{0}\}} \lambda_{\vec{k}} \; {\bf
P}^T\vec{k} e^{2\pi i \vec{k} \cdot \vec{y}} \; , \; \lambda_{\vec{k}} \in \C \Bigr \}. 
\end{equation*}
\end{proposition}
Note that the sequence $(\lambda_{\vec{k}})$ which appears in the definition of
${\mathcal L}_{\textbf{P}}$ does not necessarily belong to $l^2$. It only satisfies
$\sum_{\vec{k}\in\Z^m\setminus\{{\bf 0}\}} {\mid\lambda_{\vec{k}}\mid}^2 \; {\mid \textbf{P}^T\vec{k}\mid}^2 < +\infty
$. 

There is an alternative way to express the above result \cite{wellander2018two}. Let us define $\textbf{P}$-dependent gradient and divergence operators acting on functions defined on domains in $\R^m$. They are
\begin{definition}[\textbf{P}-gradient and \textbf{P}-divergence]\label{def:R-differetnialoperators}
\begin{equation*}
 \nabla_\textbf{P}\;u(\vec{y}) =
 \mathrm{grad}_\textbf{P}\;u(\vec{y}) :=
 \left({\mathbf P}^T \nabla_{\vec{y}}\right)  u(\vec{y}),
\end{equation*}
\begin{equation*}
\nabla_\textbf{P}\cdot \vec{u}(\vec{y}) =
\mathrm{div}_\textbf{P}\;\vec{u}(\vec{y}) :=
\left(\textbf{P}^T \nabla_{\vec{y}}\right) \cdot \vec{u}(\vec{y}).
\end{equation*}
\end{definition}

We introduce well-suited function spaces

\begin{definition}[Function spaces for \textbf{P}-gradient and \textbf{P}-divergence]
\begin{equation}
\displaystyle
{\mathcal H}_{\sharp}(\mathrm{grad}_\textbf{P},Y^m) := \displaystyle{\Bigl \{u\in
{L^2_\sharp(Y^m)} \; \mid \; \mathrm{grad}_\textbf{P}\;u \in
{L^2_\sharp(Y^m; \R^n)}
\Bigr \},
}
\label{def:grad-Yspace}
\end{equation}
\begin{equation}
\displaystyle
{\mathcal H}_{\sharp}(\mathrm{div}_\textbf{P}, Y^m) :=  \displaystyle{\Bigl \{ \vec{u}\in
{L^2_\sharp(Y^m; \R^n)} \; \mid \; \mathrm{div}_\textbf{P}\;\vec{u}  \in
{L^2_\sharp(Y^m)}
\Bigr \},
}
\label{def:divR-Yspace}
\end{equation}
\begin{equation}
\displaystyle
{\mathcal H}_{\sharp}(\mathrm{div}_{\textbf{P}_0}, Y^m) :=  \displaystyle{\Bigl \{ \vec{u}\in
{\mathcal H}_{\sharp}(\mathrm{div}_\textbf{P}, Y^m) \; \mid \; \mathrm{div}_\textbf{P}\;\vec{u} \equiv 0 \Bigr \}.
}
\label{def:div_0-Yspace}
\end{equation}
Further, we define the usual Sobolev spaces 
\begin{equation}
\displaystyle
{H^{1}}(\Omega) :=  \displaystyle{\Bigl \{ u\in {L^2(\Omega)} \; \mid \; \mathrm{grad}\;u(\vec{x}) \in
{L^2(\Omega; \R^n)}
\Bigr \}},
\label{def:grad-space}
\end{equation}
\begin{equation}
\displaystyle
{H^{1}_0}(\Omega) :=  \displaystyle{ \Bigl\{ u\in H^1(\Omega)\; \mid \;  u|_{\partial \Omega} = 0
\Bigr \},
}
\label{def:grad_0-space}
\end{equation}
\end{definition}

In the following main compactness results we let $\textbf{P}$ be a matrix with $m$ rows and $n$ columns ($m>n$)
satisfying (\ref{criterion}) and assume that $\O$ is an open bounded set
of $\R^n$ with a Lipschitz boundary.
We begin by recasting  Proposition \ref{prop:gradtwoscale} into the familiar form from \cite{wellander2018two}

\begin{proposition}\label{prop:grad-split}
Let  $(u_\e)$ be a uniformly bounded
sequence in $H^1(\Omega)$.
Then, there exists a subsequence $(u_{\e_k})$
as well as functions $u\in H^1(\Omega)$ and $\mathrm{grad}_\textbf{P} \;  {u}_1\xy \in L^2(\Omega, L^2_\sharp(Y^m; \R^n ))$,
such that
\begin{equation}\label{eq:gradsplit_2s_strong}
u_{\e_k} \dto u(\vec{x}),
\end{equation}
\begin{equation}
\mathrm{grad}\; u_{\e_k} \wdto \mathrm{grad} \; u (\vec{x}) + \mathrm{grad}_\textbf{P} \;  u_1\xy ,  \qquad
{\e_k} \to 0,
\end{equation}
\end{proposition}
\begin{proof}
The proof \cite{wellander2018two} is identical to the one for Proposition~\ref{prop:gradtwoscale} given in \cite[Proposition~2.13]{Bouchitte+etal2010}.
\end{proof}

We have the following analogue of Proposition~1.14(ii) in \cite{Allaire1992}.
\begin{proposition}\label{prop:eta_grad}
Let  $(u_\e)$ and $(\e\mathrm{grad}\;{u}_\e)$ be two uniformly bounded
sequences in $L^2(\Omega)$ and $L^2(\Omega;\R^n)$, respectively.
Then there exist a subsequence
 $({u}_{\e_k})$
and a function {${u}_0\in L^2(\O,{\mathcal H}_{\sharp}(\mathrm{grad}_\textbf{P} ,Y^m))$}
such that
\begin{equation}\label{eq:u_2s_weak}
u_{\e_k} \wdto u_0(\vec{x},\vec{y})\; \qquad \in L^2(\O,L^2_{\sharp}(Y^m)),
\end{equation}
\begin{equation}\label{eq:etagradu_2s}
{\e_k} \, \mathrm{grad}\; u_{\e_k} \wdto  \mathrm{grad}_\textbf{P}  \; {u}_0\xy  \; \in L^2(\O,L^2_{\sharp}(Y^m;\R^n)),
\end{equation}
as ${\e_k} \to 0$.
\end{proposition}

The following result from \cite{Bouchitte+etal2010,wellander2018two} is the two-scale and cut projection counterpart of the two-scale limits of divergence free bounded sequences,   \emph{e.g.} see \cite{Allaire1992} for the divergence free case. The proof uses similar ideas to Propositions~\ref{prop:grad-split} and \ref{prop:eta_grad}.

\begin{proposition}
\label{prop:eta_div_free}
Let  $(\vec{u}_\e)$  be a divergence free sequence, uniformly bounded in $L^2(\Omega;\R^n)$. Then, there exist a subsequence $(\vec{u}_{\e_k})$
and a function {$\vec{u}_0\in  L^2(\O,{\mathcal H}_{\sharp}(\mathrm{div}_{\textbf{P}_0},Y^m))$}
such that
\begin{equation*}
\vec{u}_{\e_k} \wdto \vec{u}_0(\vec{x},\vec{y}).
\end{equation*}
The function $\vec{u}_0$ satisfies
\begin{equation}
\label{eq:div_R_free_limit}
         \begin{aligned}
& \mathrm{div}_\textbf{P}\, \vec{u}_0(\vec{x},\vec{y}) = 0, \qquad \mbox{a.e. } \Omega\times Y^m,  \\
& \mathrm{div} \int_{Y^m} \vec{u}_0(\vec{x},\vec{y}) \; \mathrm{d}\vec{y} = 0,  \qquad \mbox{a.e. } \Omega.
         \end{aligned}
\end{equation}
\end{proposition}

\section{Low-Frequency Limit
}

We are interested in the limit of a sequence of spectral problems which amounts to finding $(\lambda_\eta,v_\eta)\in \R^+ \times H^1_0(\Omega)$, $v_\eta\neq 0$, such that
 \begin{equation}  \label{eq:NL_eta}
     \left\{
        \begin{aligned}
          -  \mathrm{div} \; \left( A\left( \vec{x}, \frac{\textbf{P}\vec{x}}{\e} \right) \nabla v_\eta(\vec{x}) \right ) & = {\lambda_\eta} v_\eta(\vec{x}) \;, \qquad \vec{x} \in \Omega,\\
            v_\e|_{\partial\Omega} & = 0.
        \end{aligned}
    \right.
 \end{equation}

This problem can be reframed by introducing the Green operator ${\mathcal G}_\eta$ in ${\mathcal L}(\Omega)$ defined for every $f\in L^2(\Omega)$ by ${\mathcal G}_\eta f=u_\eta$ where $u_\eta$ is the unique solution in  $H^1_0(\Omega)$ (by Lax-Milgram) of
\begin{equation}  \label{eq:GREEN_eta}
     \left\{
        \begin{aligned}
          -  \mathrm{div} \; \left( A\left( \vec{x}, \frac{\textbf{P}\vec{x}}{\e} \right) \nabla u_\eta(\vec{x}) \right )  + u_\eta(\vec{x})
          & = f(\vec{x}) \;, \qquad \vec{x} \in \Omega,\\
            u_\e|_{\partial\Omega} & = 0.
        \end{aligned}
    \right.
 \end{equation}

The Green operator ${\mathcal G}_\eta$ is self-adjoint and compact for a fixed $\eta>0$ and its spectrum consists of a countable set of eigenvalues that can be ordered (with the possibility of multiple eigenvalues) by decreasing order with zero as the accumulation point:
\begin{equation}
    \sigma({\mathcal G}_\eta)=\{0\}\cup\bigcup_{k\geq 1}\Lambda^k_\eta \; , \; \Lambda^1_\eta\geq \Lambda^2_\eta\geq
    \cdots
    \Lambda^k_\eta \cdots \to 0.
\end{equation}
Thus, the spectrum of (\ref{eq:GREEN_eta}) is discrete and consists of a countable set of eigenvalues converging to zero. Additionally, there is a normalized eigenfunction $u^k_\eta$ in $L^2(\Omega)$ associated with each eigenvalue $\Lambda_\eta^k$, such that the family $\{u^k_\eta\}_{k\geq 1}$ forms an orthonormal basis of $L^2(\Omega)$. It is well known \cite{bensoussan2011asymptotic} that the  eigenvalues $\Lambda_\eta^k$ satisfy $\Lambda_\eta^k=(\lambda_\eta^k+1)^{-1}$ and thus
the spectrum of (\ref{eq:NL_eta}) is discrete and consists of a countable set of strictly positive eigenvalues converging to infinity. 

We would like to analyse the limit of the spectrum $\sigma_\eta$ of (\ref{eq:NL_eta}) when $\eta$ tends to zero. We shall assume in this section that $k$ is fixed. To do so, we need to pass to the limit in (\ref{eq:GREEN_eta}), which is established in the following Proposition

\begin{proposition}\label{maintheorem_Homogenized_spectrum}
The sequence of solutions  $(u_\e)$ of the sequence of problems (\ref{eq:GREEN_eta}) converges weakly in $H^1_0(\Omega)$  to the solution $u$ of the homogenized problem
 \begin{equation}  \label{eq:homogenized_static_strong}
     \left\{
        \begin{aligned}
   -  \mathrm{div} \;  A^h(\vec{x})  \nabla  u(\vec{x})
   + u(\vec{x})
   & =    f (\vec{x})  \;, \qquad \vec{x} \in \Omega,\\
            u|_{\partial\Omega} & = 0,
        \end{aligned}
        \right.
\end{equation}
where

\begin{equation}  \label{eq:homogenized_sigma_static}
     A^h_{ik}(\vec{x}) =   \int_{Y^m} A_{ij}\left(\vec{x},\vec{y} \right) \left( \delta_{jk} -  \nabla_{\textbf{P}}\chi^k (\vec{x},\vec{y})\right) \mathrm{d}\vec{y},
\end{equation}
and $\nabla_{\textbf{P}}\chi^k \in \mathcal{C}^{0,\alpha}(\O,{\mathcal L}_\textbf{P})$
solves the local equation
\begin{equation}\label{eq:local_static}
    \int_{Y^m} A_{ij}\left(\vec{x},\vec{y} \right) \left( \delta_{jk} -  \nabla_{\textbf{P}}\chi^k (\vec{x},\vec{y})\right)  \nabla_\textbf{P} \phi(\vec{y}) \; \mathrm{d}\vec{y}=0.
\end{equation}
\end{proposition}

\begin{remark}
In proposition \ref{maintheorem_Homogenized_spectrum} and the analysis that follows, we assume Einstein's summation convention over repeated indices.  
\end{remark}
\begin{proof}
Choose $\phi \in D(\Omega;\mathcal{C}^\infty_\sharp(Y^m))$ and let
$\phi_\e(x):= \phi(x,\frac{{\textbf{P}} \vec{x}}{\e})$. Testing with $\e \phi_\e$ gives
 \begin{equation*}  \label{eq:weak_static_eps}
    \int_\Omega A\left(\vec{x},\frac{\textbf{P} \vec{x}}{\e}\right)  \nabla u_\e(\vec{x}) \cdot \left(\e \nabla \phi_\e(\vec{x}) + \nabla_\textbf{P} \phi_\e(\vec{x})\right) \; \mathrm{d}\vec{x}
    +\int_\Omega u_\e(\vec{x}) \e\phi_\e(\vec{x}) \mathrm{d}\vec{x}
    = \int_\Omega f (\vec{x}) \e  \phi_\e(\vec{x})\; \mathrm{d}\vec{x}.
\end{equation*}
Sending $\e \to 0$ gives due to Proposition~\ref{prop:grad-split}
 \begin{equation*}  \label{eq:locallimit_static}
    \int_\Omega  \int_{Y^m} A\left(\vec{x},\vec{y} \right) \left(\nabla u(\vec{x}) +  \nabla_\textbf{P} u_1(\vec{x}, \vec{y})\right) \cdot \nabla_\textbf{P} \phi(\vec{x}, \vec{y}) \; \mathrm{d}\vec{x} \mathrm{d}\vec{y}= 0.
\end{equation*}
Substituting $ \nabla_\textbf{P} u_1(\vec{x}, \vec{y}) = - \nabla_\textbf{P} \chi^k (\vec{x},\vec{y})\partial_{k}u(\vec{x})$ yields the local equation
\begin{equation*}  \label{eq:local_static_proof}
    \int_{Y^m} A_{ij}\left(\vec{x},\vec{y} \right) \left( \delta_{jk} -  \nabla_{\textbf{P}}\chi^k (\vec{x,}\vec{y})\right)\cdot \nabla_{\textbf{P}} \phi(\vec{x},\vec{y}) \; \mathrm{d}\vec{y}=0.
\end{equation*}
The homogenized equation is obtained by choosing a test function $\phi \in D(\Omega)$. We get analogously
 \begin{equation*}  \label{eq:globallimit_static}
    \begin{split}
    \int_\Omega  \int_{Y^m} A\left(\vec{x},\vec{y} \right) \left( \nabla u(\vec{x}) +  \nabla_\textbf{P} u_1(\vec{x}, \vec{y})\right) \cdot \nabla \phi(\vec{x}) \; \mathrm{d}\vec{x} \mathrm{d}\vec{y}
    &+ \int_\Omega  \int_{Y^m} u(\vec{x})  \phi(\vec{x}) \; \mathrm{d}\vec{x} \mathrm{d}\vec{y}
    \\&\qquad\qquad
    =   \int_\Omega  \int_{Y^m} f (\vec{x})  \phi(\vec{x}) \; \mathrm{d}\vec{x} \mathrm{d}\vec{y}
    \end{split}
\end{equation*}
That is
 \begin{equation*}  \label{eq:homogenized_static}
    \int_\Omega  A^h(\vec{x})  \nabla  u(\vec{x})  \cdot \nabla \phi(\vec{x}) \; \mathrm{d}\vec{x}
    +\int_\Omega   u (\vec{x})  \phi(\vec{x}) \; \mathrm{d}\vec{x}
    =   \int_\Omega   f (\vec{x})  \phi(\vec{x}) \; \mathrm{d}\vec{x},
\end{equation*}
where
\begin{equation*}
          A^h_{ik}(\vec{x}) =   \int_{Y^m} A_{ij}\left(\vec{x},\vec{y} \right) \left( \delta_{jk} -  \nabla_{\textbf{P}}\chi^k (\vec{x}, \vec{y})\right) \mathrm{d}\vec{y}.
\end{equation*}
The local equation \eqref{eq:local_static} has a bounded gradient $\nabla_\textbf{P}\chi^k $ which belongs to ${\mathcal L}_\textbf{P}$, by a standard argument using Lax-Milgram. The regularity in $\vec{x}$ follows from
\begin{equation*}
    \begin{aligned}
        \gamma\lVert \nabla_{\textbf{P}}\chi^k(\vec{x}_1,\cdot) - \nabla_{\textbf{P}}\chi^k(\vec{x}_2,\cdot)\rVert^2 &\leq \int_{Y^m}A(\vec{x}_1,\vec{y}) (\nabla_{\textbf{P}}\chi^k(\vec{x}_1,\vec{y}) - \nabla_{\textbf{P}}\chi^k(\vec{x}_2,\vec{y}))\\
        &\hspace{3cm}\cdot \overline{(\nabla_{\textbf{P}}\chi^k(\vec{x}_1,\vec{y} - \nabla_{\textbf{P}}\chi^k(\vec{x}_2,\vec{y}))}dy\\
        &\leq\lVert A(\vec{x}_1,\cdot)-A(\vec{x}_2,\cdot)\rVert_{L^\infty}\lVert \nabla_{\textbf{P}}\chi^k(\vec{x}_1,\cdot) - \nabla_{\textbf{P}}\chi^k(\vec{x}_2,\cdot)\rVert\\
        &\hspace{3cm}\cdot(1 + \lVert \nabla_{\textbf{P}}\chi^k(\vec{x}_1,\cdot)\rVert)\\
        &\leq [A]_{\mathcal{C}^{0,\alpha}} \lvert\vec{x}_1-\vec{x}_2\rvert^\alpha\lVert \nabla_{\textbf{P}}\chi^k(\vec{x}_1,\cdot) - \nabla_{\textbf{P}}\chi^k(\vec{x}_2,\cdot)\rVert\\
        &\hspace{3cm}\cdot (1 + \lVert \nabla_{\textbf{P}}\chi^k(\vec{x}_1,\cdot)\rVert).
    \end{aligned}
\end{equation*}
\end{proof}




Thanks to the compact embedding of $H^1_0(\Omega)$ into $L^2(\Omega)$, we deduce from proposition \ref{maintheorem_Homogenized_spectrum} that for $f\in L^2(\Omega)$, the sequence of solutions  $(u_\e)$ of (\ref{eq:GREEN_eta}) converges strongly in $L^{2}(\Omega)$ to the unique solution $u$ of the homogenized problem (\ref{eq:homogenized_static_strong})-(\ref{eq:local_static}). It follows that the sequence of Green operators $({\mathcal G}_\eta)$ converges in norm in $L^2(\Omega)$ to the Green operator ${\mathcal G}$ of the homogenized problem defined for every $f\in L^2(\Omega)$ by ${\mathcal G} f=u$, where $u$ is the unique solution in $H^1_0(\Omega)$ of the homogenized problem (\ref{eq:homogenized_static_strong})-(\ref{eq:local_static}).

The Green operator ${\mathcal G}$ is self-adjoint and compact since the homogenized parameter $A^h_{ik}$ inherits the properties of the matrix $A_{ij}$, see Assumption \ref{ass:A}. So, for $\eta=0$, the spectrum associated with (\ref{eq:homogenized_static_strong})-(\ref{eq:local_static}) is discrete and consists of a countable set of eigenvalues converging to zero, with zero as the accumulation point:
\begin{equation}
    \sigma({\mathcal G})=\{0\}\cup\bigcup_{k\geq 1}\Lambda^k \; , \; \Lambda^1\geq \Lambda^2\geq
    \cdots
    \Lambda^k \cdots \to 0.
\end{equation}
Further, there is a normalized eigenfunction $u^k$ in $L^2(\Omega)$ associated with each eigenvalue, such that the family ${(u^k)}_{k\geq 1}$ forms an orthonormal basis of $L^2(\Omega)$.

It can be shown using the min-max principle that
for a fixed integer $k$, $| \Lambda_\eta^k-\Lambda^k | \leq \Vert {\mathcal G}_\eta - {\mathcal G} \Vert_{{\mathcal L}(\Omega)}$ and thus the convergence in norm of ${\mathcal G}_\eta$ to ${\mathcal G}$ yields the convergence of each individual eigenvalue $\Lambda_\eta^k$ ordered by decreasing order to $\Lambda^k$. Therefore, we obtain the following result

\begin{theorem}\label{thm: SpectralhomogenizedLimit}
As $\e$ goes to zero, the spectrum $\sigma_\e$ of (\ref{eq:NL_eta}) converges pointwise to $\sigma^h=\{\Lambda^{-1}-1 \mid \Lambda\in \sigma({\mathcal G})\}$.
\end{theorem}

The proof follows closely that of Theorem 2.2 in \cite{allaire1998bloch} with Proposition 2.5 in that work replaced by our proposition \ref{maintheorem_Homogenized_spectrum}.
A key inequality in the proof is
\begin{equation}\label{bleue}
   \Vert {\mathcal G}_\eta - {\mathcal G} \Vert_{{\mathcal L}(\Omega)} =  \sup_{{\Vert f \Vert}_{L^2(\Omega)}=1} \Vert{\mathcal G}_\e f-{\mathcal G}f\Vert_{L^2(\Omega)} \leq\Vert{\mathcal G}_\e f_\e-{\mathcal G}f\Vert_{L^2(\Omega)} + \Vert{\mathcal G} f_\e-{\mathcal G}f\Vert_{L^2(\Omega)} + \e,
\end{equation}
where $f_\e$ is an $\e$ minimizer.
Let us pass to the limit when $\e$ goes to zero.
Since the sequence $(f_\e)$ is bounded in $L^2(\Omega)$, $f_\e\rightharpoonup f$ weakly in $L^2(\Omega)$ up to a subsequence. Moreover, ${\mathcal G}_\e f_\e\to{\mathcal G} f$ strongly in $L^2(\Omega)$ thanks to proposition \ref{maintheorem_Homogenized_spectrum} and ${\mathcal G}f_\e\to{\mathcal G} f$ strongly in $L^2(\Omega)$ since ${\mathcal G}$ is a compact operator, and thus passing to the limit in (\ref{bleue}) we get $\Vert {\mathcal G}_\eta - {\mathcal G} \Vert_{{\mathcal L}(\Omega)}\to 0$.

\begin{remark}
Consider a rod with a quasiperiodic conductivity distribution, $\sigma$, generated by a cut-and-projection of an m-dimensional unit cell $Y^m$ with a $\textbf{P}$ satisfying \eqref{criterion}. The homogenized conductivity of the rod defined by  \eqref{eq:homogenized_sigma_static} is given by  the harmonic mean over $Y^m$ and thus (\ref{eq:homogenized_static_strong}) reduces to
\begin{equation}  \label{eq:GREEN1D_0}
     \left\{
        \begin{aligned}
   -  \frac{d}{dx} \;  {{\left(\int_{Y^m}
A^{-1}(x,\vec{y})d\vec{y}\right)}^{-1}}  \frac{d}{dx}  u(x)  + u(x)
& =    f(x)  \;, \qquad x \in ]0,1[,\\
            u(0)=u(1) & = 0.
        \end{aligned}
        \right.
\end{equation}
\end{remark}

The fact that $\lambda_\eta^k\to\lambda^k$ when $\e\to 0$ for any fixed integer $k$ does not inform us about the asymptotic behaviour of sequences of eigenvalues $\lambda_\eta^{k(\e)}$ where $k(\e)$ tends to infinity. In \cite{allaire1998bloch} different asymptotic regimes have been studied in the periodic case depending upon whether the limit of $\e^{2}\lambda_\e$ is finite (critical scale), infinite (super-critical scale) or zero (sub-critical scale). We shall follow a similar route for the quasi-periodic setting.  

In the following section, we shall consider the critical scale i.e. eigenvalues of the order $\eta^{-2}$ for the problem (\ref{eq:NL_eta}). The physical interpretation of this model is based on the property of the possible appearance of band gaps in the spectrum, i.e. frequency intervals for which no wave can propagate within the medium. This problem has been addressed in the periodic case in \cite{allaire1998bloch}. 

\section{High-Frequency Limit}
In this section, we outline a different approach to the high-frequency limit of the spectrum $\sigma_\eta$. While the ideas are very much similar to those of \cite{allaire1998bloch}, a simple change of variables transforms the problem to a much more natural domain. This rescaling was originally introduced for spectral problems with periodic coefficients by Conca  and Vanninathan \cite{concaSpectral1988}
and further developed by Allaire and Conca \cite{allaire1998boundary}, where it was combined with Bloch wave decomposition. Here, we show that the rescaling alone, together with pseudo-eigenfunction arguments, suffices to obtain the spectral decomposition---and, crucially, extends to the quasiperiodic setting where Bloch decomposition is not available.

Consider the problem of finding $(\lambda_\eta,v_\eta)$ satisfying
\begin{equation}\label{eq: LowFrequency}
    \begin{cases}
        \begin{aligned}
          -  \mathrm{div} \; \left( A\left( \vec{x}, \frac{\textbf{P} \vec{x}}{\eta} \right) \nabla v_\eta({\vec x}) \right ) & = \lambda_\eta v_\eta({\vec x}) \;, \qquad \vec{x} \in \Omega\\
            v_\eta|_{\partial\Omega} & = 0.
        \end{aligned}
    \end{cases}
\end{equation}
We are interested in the limit
\begin{equation}
    \lim_{\eta\rightarrow0} a_\eta^{2}\sigma_\eta,
\end{equation}
where $a_\eta$ is a sequence tending to zero. To this end, we rewrite equation \eqref{eq: LowFrequency} as
\begin{equation}\label{eq: HighFrequencyClassical}
    \begin{cases}
        \begin{aligned}
          -  a_\eta^2\mathrm{div} \; \left( A\left( \vec{x}, \frac{\textbf{P} \vec{x}}{\eta} \right) \nabla v_\eta({\vec x}) \right ) & = \tilde{\lambda}_\eta v_\eta({\vec x}) \;, \qquad \vec{x} \in \Omega\\
            v_\eta|_{\partial\Omega} & = 0.
        \end{aligned}
    \end{cases}
\end{equation}
If we denote the spectrum of the above as $\tilde{\sigma}_\eta$, then it is easy to see that $a_\eta^{2}\sigma_\eta = \tilde{\sigma}_\eta$. Hence we may study equation \eqref{eq: HighFrequencyClassical} instead. This is the approach taken in \cite{allaire1998bloch}.

\begin{figure}
    \centering
    \includegraphics[width=0.9\linewidth]{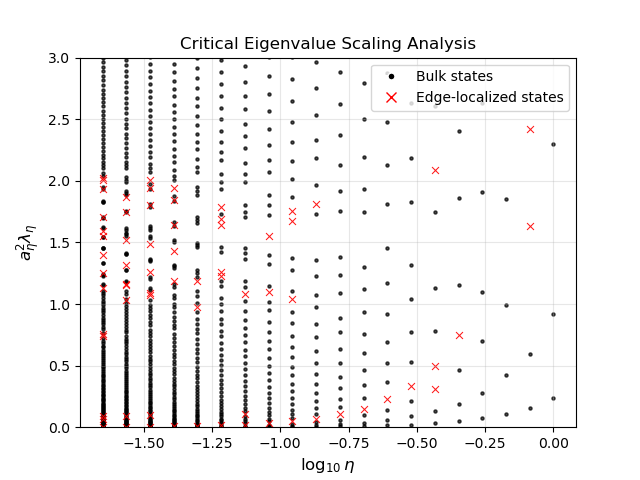}
    \caption{Eigenvalues in the critical scaling regime for $A(x,y,z) = 2 + \cos(x)^2 + \cos (y)^2\cos(z)$ along the cut $(1,\tau)$, where $\tau$ is the golden ratio.}
    \label{fig:eigenvalue_critical}
\end{figure}
Pretend for a second that $A$ is periodic in its second variable. Then, as $\eta\rightarrow0$, we can think of squeezing more and more fundamental cells into the domain $\Omega$. This picture is captured by equation \eqref{eq: HighFrequencyClassical}. However, we can also adopt a different view. Namely, instead of shrinking the size of our fundamental cells and keeping $\Omega$ fixed, we may also think of expanding $\Omega$ and keeping the size of the fundamental cells fixed. Formally, this is represented by the change of variables $\vec y  = a_\eta^{-1}\vec x$. The key observation is now the following: suppose $(\tilde{\lambda}_\eta,v_\eta)$ is a solution to \eqref{eq: HighFrequencyClassical}. Let $\tilde{v}(\vec y) = v(a_\eta\vec y)$ for $\vec y \in a_\eta^{-1}\Omega$ and observe that
\begin{align*}
    &-\mathrm{div}_{\vec y}\left(A(a_\eta \vec y, \frac{a_\eta}{\eta}\textbf{P}\vec y) \nabla_{\vec y} \tilde{v}(\vec{y})\right) = -  \partial_{i}(A_{ik}(a_\eta \vec y, \frac{a_\eta}{\eta}\textbf{P}\vec y)\partial_{k}\tilde{v}(\vec{y}))\\   
    &\quad =-  a_\eta^2 (\partial_{1,i}A_{ik})(a_\eta \vec y, \frac{a_\eta}{\eta}\textbf{P}\vec y)(\partial_kv)(a_\eta\vec y)
    +  \frac{a_\eta^2}{\eta}  P_{ji}(\partial_{2,j}A_{ik})(a_\eta \vec y, \frac{a_\eta}{\eta}\textbf{P}\vec y)(\partial_kv)(a_\eta\vec y)\\
    &\quad \qquad + a_\eta^2 A_{ik}(a_\eta \vec y, \frac{a_\eta}{\eta}\textbf{P}\vec y)(\partial_i\partial_k v)(a_\eta \vec y)\\
    &\quad =-  a_\eta^2 (\partial_{1,i}A_{ik})(\vec x, \frac{\textbf{P}\vec x}{\eta})(\partial_kv)(\vec x)
    +  \frac{a_\eta^2}{\eta}  P_{ji}(\partial_{2,j}A_{ik})(\vec x, \frac{\textbf{P}\vec x}{\eta})(\partial_kv)(\vec x)\\
    &\quad \qquad + a_\eta^2 A_{ik}(\vec x, \frac{\textbf{P}\vec x}{\eta})(\partial_i\partial_k v)(\vec x)\\
    &\qquad =-a_\eta^{2}\mathrm{div} \; \left( A\left( \vec{x}, \frac{\textbf{P} \vec{x}}{\eta} \right) \nabla v_\eta({\vec x}) \right )= \tilde{\lambda}_\eta v(\vec x)=\tilde{\lambda}_\eta \tilde{v}(\vec y).
\end{align*}

Hence, $(\tilde{\lambda}_\eta, \tilde{v})$ is a solution of
\begin{equation}\label{eq: HighFrequencyNew}
    \begin{cases}
        \begin{aligned}
          - \mathrm{div} \; \left( A\left( a_\eta \vec{y}, \frac{a_\eta}{\eta}\textbf{P} \vec{y} \right) \nabla \tilde{v}_\eta({\vec y}) \right ) & = \tilde{\lambda}_\eta\tilde{v}_\eta({\vec y}) \;, \qquad \vec{y} \in a_\eta^{-1}\Omega\\
            \tilde{v}_\eta|_{\partial a_\eta^{-1}\Omega} & = 0.
        \end{aligned}
    \end{cases}
\end{equation}
The main advantage of equation \eqref{eq: HighFrequencyNew} as opposed to \eqref{eq: HighFrequencyClassical} becomes most apparent when $a_\eta = \eta$, the \textit{critical scaling}. In this case, the scale of oscillations in the fast variable remains fixed and locally the influence of the slow variable becomes increasingly negligible. From this, we can explicitly construct approximate eigenfunctions for values in the Bloch spectrum, without having to construct complicated projection and extension operators as in \cite{allaire1998bloch}. More importantly, this allows us to treat the quasiperiodic setting as well. It will turn out that the critical scaling is precisely the asymptotic regime where the spectrum has non-trivial gap structure in the limit; an example of this is shown in Figure~\ref{fig:eigenvalue_critical}.

Before turning to the individual scalings, we first have to settle some technicalities concerning the definition of $\sigma$. We will be working with operators of the form $T\colon H^1_0\to H^{-1}$
\[
  \langle Tu,v\rangle_{H^{-1},H^1}=\int B(\vec y)\,\nabla u(\vec y)\cdot
  \overline{\nabla v(\vec y)}\,d\vec y,
\]
for a Hermitian, uniformly elliptic coefficient matrix $B$ obtained from $A$ by
freezing and rescaling. By uniform ellipticity (Assumption~\ref{ass:A}(iii)) this form
is closed and bounded below as a form in $L^2$ with domain $H^1_0$, and it is symmetric by
Assumption~\ref{ass:A}(i). The first representation theorem
\cite[Thm.~VI-2.1]{kato1966perturbation} therefore associates to it a unique
self-adjoint operator $\mathcal T\colon D(\mathcal T)\subset L^2\to L^2$ characterised by
$\langle Tu,v\rangle_{H^{-1},H^1}=\langle \mathcal T u,v\rangle_{L^2}$ for all $v\in H^1_0$ on
the domain $D(\mathcal T)=\{u\in H^1_0 : Tu\in L^2\}$. We thus write $\sigma(T):=\sigma(\mathcal T)$. Moreover, $(\mathcal T+1)^{-1}$ agrees with
the restriction to $L^2$ of $(T+1)^{-1}\colon H^{-1}\to H^1_0$. 

The following lemma will enable us to turn suitable pseudo-eigenfunctions of $T$ into estimates on the spectrum $\sigma$:


\begin{lemma}\label{lem: spectral_equivalence}
Let $T$ and $\mathcal T$ be as above, and let $\lambda\ge 0$. Suppose there
exists $u_\epsilon\in H^1_0\setminus\{0\}$ with
\begin{equation}
  \lVert (T-\lambda)u_\epsilon\rVert_{H^{-1}}\le\epsilon\,
  \lVert u_\epsilon\rVert.
\end{equation}
Then, there exists some $C>0$ depending only on the ellipticity constant $\gamma$ and
$\lVert A\rVert_{L^\infty}$ such that for $\epsilon$ small enough
\begin{equation}
  d\big(\lambda,\sigma(T)\big)\leq C(1+\lambda)\epsilon.
\end{equation}
\end{lemma}

\begin{proof}
Set $r_\epsilon:=(T-\lambda)u_\epsilon\in H^{-1}$, so that
$\lVert r_\epsilon\rVert_{H^{-1}}\leq\epsilon\lVert u_\epsilon\rVert_{H^1}$.
From $(T+1)u_\epsilon=(\lambda+1)u_\epsilon+r_\epsilon$ in $H^{-1}$, applying
$(T+1)^{-1}\colon H^{-1}\to H^1_0$ and using that this resolvent restricts on
$L^2$ to $(\mathcal T+1)^{-1}$, we obtain
\begin{equation}
  (\mathcal T+1)^{-1}u_\epsilon
  =(\lambda+1)^{-1}(u_\epsilon-(T+1)^{-1}r_\epsilon).
\end{equation}
Since $(T+1)^{-1}$ is bounded $H^{-1}\to H^1_0$ with norm depending only on
$\gamma$ and $\lVert A\rVert_{L^\infty}$,
\begin{equation}
  \lVert[(\mathcal T+1)^{-1}-(\lambda+1)^{-1}]u_\epsilon
  \rVert
  =(\lambda+1)^{-1}\lVert (T+1)^{-1}r_\epsilon\big\rVert
  \le\frac{C\,\epsilon}{\lambda+1}\,\lVert u_\epsilon\rVert.
\end{equation}
The operator $(\mathcal T+1)^{-1}$ is bounded and self-adjoint, thus the resolvent identity yields
\begin{equation}
  d((\lambda+1)^{-1},\,\sigma((\mathcal T+1)^{-1}))
  \le \frac{C\,\epsilon}{(1+\lambda)}.
\end{equation}
Using that
$\sigma\big((\mathcal T+1)^{-1}\big)
=\{(\tilde\lambda+1)^{-1}:\tilde\lambda\in\sigma(\mathcal T)\}$, we obtain for $\epsilon$ small enough
\begin{equation}
  d(\lambda,\,\sigma(\mathcal T))
  \leq 2C(1+\lambda)\epsilon,
\end{equation}
from which the claim follows.
\end{proof}

\subsection{Critical Scaling}\label{sec:critical}
In the critical scaling, when $a_\eta = \eta$, equation \eqref{eq: HighFrequencyNew} simplifies to
\begin{equation}\label{eq: HighFrequencyCritical}
    \begin{cases}
        \begin{aligned}
          - \mathrm{div} \; \left( A\left( \eta \vec{y},\textbf{P} \vec{y} \right) \nabla \tilde{v}_\eta({\vec y}) \right ) & = \tilde{\lambda}_\eta\tilde{v}_\eta({\vec y}) \;, \qquad \vec{y} \in \eta^{-1}\Omega,\\
            \tilde{v}_\eta|_{\partial \eta^{-1}\Omega} & = 0.
        \end{aligned}
    \end{cases}
\end{equation}
We begin by proving the inclusion of the Bloch spectrum. To this end, for $\vec x_0 \in \Omega$ we will first need to define the operators
 \begin{equation}
     T_{\vec{x}_0}: H^1(\mathbb{R}^n)\rightarrow H^{-1}(\mathbb{R}^n), u\mapsto \langle T_{\vec{x}_0}u,v \rangle_{H^{-1},H^{1}}:=\int_{\mathbb{R}^n}A\left( \vec{x}_0, \textbf{P}\vec{y}\right)\nabla u(y) \cdot \nabla \overline{v}(y) dy.
 \end{equation}
 \begin{equation}
     T_\eta: H^1_0(\eta^{-1}\Omega)\rightarrow H^{-1}(\eta^{-1}\Omega), u\mapsto \langle T_{\eta}u,v \rangle_{H^{-1},H^{1}}:=\int_{\eta^{-1}\Omega}A\left( \eta\vec{y}, \textbf{P}\vec{y}\right)\nabla u(y) \cdot \nabla \overline{v}(y) dy .
 \end{equation}

\begin{proposition}\label{prop: BlochSpectrum}
    It holds that
    \begin{equation}
    \lim_{\eta\rightarrow 0}\eta^{2}\sigma_\eta \supset \sigma_{\text{Bloch}}:= \bigcup_{\vec{x}_0\in \bar{\Omega}}\sigma(T_{\vec{x}_0}).
    \end{equation}
\end{proposition}
\begin{proof}
Let $\vec{x_0}\in\overline{\Omega}$ and $\lambda_{\vec{x}_0}\in\sigma(T_{\vec{x}_0})$. Because $A$ is continuous in the first variable, we may assume without loss of generality that $\vec{x_0}$ lies in the interior of $\Omega$. Then for any $\epsilon>0$ there exist $\epsilon$-pseudo-eigenfunctions $u_\epsilon$ supported in $B_R(\vec{0})$ for some $R>0$, satisfying $\lVert T_{\vec{x}_0}u_\epsilon - \lambda_{\vec{x}_0}u_\epsilon\rVert_{H^{-1}} < \epsilon\lVert u_\epsilon\rVert$.

Since the map $\vec{z}\mapsto A(\vec{x}_0, \textbf{P}\vec{z})$ is almost periodic on $\R^n$ (cf.\ \cite{shubin1978almost}), for any $\epsilon'>0$ the set of $\epsilon'$-almost periods is relatively dense: there exists $L=L(\epsilon')>0$ such that every ball of radius $L$ in $\R^n$ contains an $\epsilon'$-almost period. In particular, we may choose $\vec{s}_\eta\in\R^n$ with $\lvert\vec{s}_\eta - \eta^{-1}\vec{x}_0\rvert \leq L$ such that
\begin{equation}
    \sup_{\vec{z}\in\R^n}\lvert A(\vec{x}_0, \textbf{P}\vec{z} + \textbf{P}\vec{s}_\eta) - A(\vec{x}_0, \textbf{P}\vec{z})\rvert < \epsilon'.
\end{equation}
Define $v_\eta(\vec{y}) = u_\epsilon(\vec{y} - \vec{s}_\eta)$, which is supported in $B_R(\vec{s}_\eta)\subset\eta^{-1}\Omega$ for $\eta$ sufficiently small (since $\lvert\eta\vec{s}_\eta - \vec{x}_0\rvert \leq \eta L\to 0$ and $\vec{x}_0\in\Omega$).
On $\operatorname{supp}(v_\eta)$, substituting $\vec{z}=\vec{y}-\vec{s}_\eta$, we write
    \begin{equation}
       \begin{aligned}
            \langle T_{\eta}v_\eta,v \rangle_{H^{-1},H^{1}}=\int_{B_R(\vec{0})}A\left( \eta\vec{z} + \eta\vec{s}_\eta, \textbf{P}\vec{z} + \textbf{P}\vec{s}_\eta\right)\nabla u_\varepsilon(\vec{z}) \cdot \nabla \overline{v}(\vec{z} + \vec{s}_\eta) dz .
       \end{aligned}
    \end{equation}
    Comparing this with $T_{\vec{x}_0}v_\eta(\vec{z}) $, the difference between the two coefficients can be decomposed as
    \begin{equation}
        \begin{aligned}
            &A(\eta\vec{z}+\eta\vec{s}_\eta, \textbf{P}\vec{z}+\textbf{P}\vec{s}_\eta) - A(\vec{x}_0, \textbf{P}\vec{z})\\
            &\quad= \underbrace{\bigl(A(\eta\vec{z}+\eta\vec{s}_\eta, \textbf{P}\vec{z}+\textbf{P}\vec{s}_\eta) - A(\vec{x}_0, \textbf{P}\vec{z}+\textbf{P}\vec{s}_\eta)\bigr)}_{=:\,I} + \underbrace{\bigl(A(\vec{x}_0, \textbf{P}\vec{z}+\textbf{P}\vec{s}_\eta) - A(\vec{x}_0, \textbf{P}\vec{z})\bigr)}_{=:\,II}.
        \end{aligned}
    \end{equation}
    For $\vec{z}\in B_R(\vec{0})$, the first argument satisfies $\lvert \eta\vec{z}+\eta\vec{s}_\eta - \vec{x}_0\rvert \leq \eta R + \eta L$, so by Assumption~\ref{ass:A}(iv),
    \begin{equation}
        \lVert I\rVert_{L^\infty} \leq[A]_{\mathcal{C}^{0,\alpha}}\eta^\alpha(R+L)^\alpha.
    \end{equation}
    The term $II$ is bounded by $\epsilon'$ by the choice of $\vec{s}_\eta$.

    Using the fact that $u_\epsilon$ is an $\epsilon$-pseudo-eigenfunction,  we therefore obtain 
    \begin{equation}
        \lvert \langle (T_{\eta}-\lambda )v_\eta,v \rangle_{H^{-1},H^{1}}\rvert \leq \left(\varepsilon + \varepsilon' + [A]_{\mathcal{C}^{0,\alpha}}\eta^\alpha(R+L)^\alpha\right)\lVert u_\varepsilon\rVert_{H^1}\lVert v\rVert_{H^1}.
    \end{equation}
    Finally, using ellipticity, one can show the following upper bound on 
    \begin{equation}
        \lVert u_\epsilon\rVert_{H^1}^2\leq (1 + (\lambda + \epsilon)(\gamma-\epsilon)^{-1})\lVert u_\epsilon \rVert^2,
    \end{equation}
    the claim follows upon an application of Lemma \ref{lem: spectral_equivalence}.
\end{proof}
In order to show the converse inclusion, we require a rescaled distance function. Assume that $\Omega$ has at the very least Lipschitz boundary, define

\begin{equation} d_\eta(\vec{y})=\eta^{-1}\inf_{\vec{x}\in\partial\Omega}d(\vec x,\eta\vec{y}),\quad \vec{y}\in \eta^{-1}\Omega.
\end{equation}
Then, in the spirit of \cite{allaire1998bloch}, we define the boundary spectrum as being composed of eigenfunctions that concentrate near to the boundary
\begin{equation}\label{eq: BoundarySpectrumDef}
    \sigma_{\text{Boundary}}:=\left\{\lambda\mid\exists\left(\lambda_\eta,u_\eta\right) \text{ solution of \eqref{eq: HighFrequencyNew} s.t.\ }\lambda_\eta\rightarrow\lambda \text{ and } \sup_\eta \lVert d_\eta u_\eta\rVert < \infty\right\}.
\end{equation}

\begin{theorem}\label{thm: SpectralLimit}
    The limit spectrum satisfies, in the Hausdorff sense,
    \begin{equation}        \lim\eta^{2}\sigma_\eta=\sigma_{\text{Bloch}}\cup\sigma_{\text{Boundary}}.
    \end{equation}
\end{theorem}
The proof idea is similar to the original work by Allaire and Conca \cite{allaire1998bloch}. However, rescaling allows us to work with approximate eigenfunctions instead of distributions. Since the operators are self-adjoint, we can use the language of pseudo-spectra to draw conclusions about the limiting behaviour.
\begin{lemma}\label{lem: PseudoeigenvectorsGeneral}
    Let $(\lambda_\eta, u_\eta)$ be a sequence of eigenvalues and unit eigenvectors for $T_\eta$. Assume that $\lambda_\eta\rightarrow\lambda$ and that $\lambda\notin \sigma_{\text{Boundary}}$. Then there exists a sequence of pseudo-eigenvectors $v_\eta$ such that the following hold
    \begin{itemize}
        \item $\operatorname{supp}(v_\eta)\subset\eta^{-1}\bar{\Omega}$,
        \item $v_\eta\in H^{1}$ is a weak solution in $\mathbb{R}^n$ of 
         \begin{equation} 
            \left\{
                \begin{aligned}
                    - \mathrm{div} \; \left( A\left( \eta \vec{y}, \textbf{P} \vec{y} \right) \nabla v_\eta({\vec y}) \right )& = \lambda v_\eta(\vec{y})+r_\eta \;, \qquad \vec{y} \in \eta^{-1}\bar{\Omega} \\
                    v_\eta(\vec{y})& = 0,\qquad \vec{y}\in\partial\eta^{-1}\Omega,
                \end{aligned}
            \right.
        \end{equation}
        where $r_\eta\in H^{-1}(\R^n)$ is an error term with norm converging to $0$.
    \end{itemize}
\end{lemma}
\begin{proof}
    Since $\lambda\in\lim_{\eta\rightarrow0}\eta^{2}\sigma\setminus\sigma_{\text{Boundary}}$, there exists a sequence $\eta$ such that 
    \begin{equation}
         \lambda_\eta\rightarrow\lambda \quad \text{and} \quad \lim_\eta\lVert d_\eta u_\eta\rVert = \infty.
    \end{equation}
    Consider the sequence of functions
    \begin{equation}
        v_\eta = d_\eta u_\eta.
    \end{equation}
    We have
    \begin{equation}
        \begin{aligned}
            \langle T_{\eta}v_\eta,v \rangle_{H^{-1},H^{1}}&=\int_{\eta^{-1}\Omega}A\left( \eta\vec{y}, \textbf{P}\vec{y} \right)\nabla v_\eta(\vec{y}) \cdot \nabla \overline{v}(\vec{y}) dy\\
            &=\int_{\eta^{-1}\Omega}A\left( \eta\vec{y}, \textbf{P}\vec{y} \right)\left((\nabla d_\eta)(\vec{y}) u_\eta(\vec{y}) + d_\eta(\vec{y})(\nabla u_\eta)(\vec{y})\right)\cdot \nabla \overline{v}(\vec{y}) dy\\
            &= \int_{\eta^{-1}\Omega}A\left( \eta\vec{y}, \textbf{P}\vec{y} \right)\nabla u_\eta(\vec{y}) \cdot \nabla (d_\eta\overline{v}(\vec{y}))\\
             & + A\left( \eta\vec{y}, \textbf{P}\vec{y} \right)\left ((\nabla d_\eta)(\vec{y}) u_\eta(\vec{y})\cdot\nabla \overline{v}(\vec{y}) - \nabla u_\eta(\vec{y})\cdot(\overline{v}(\vec{y})(\nabla d_\eta)(\vec{y})))\right)dy\\
             &=\int_{\eta^{-1}\Omega}\lambda_\eta d_\eta u_\eta(\vec{y})\overline{v}(\vec{y})\\ & 
             + A\left( \eta\vec{y}, \textbf{P}\vec{y} \right)\left ((\nabla d_\eta)(\vec{y}) u_\eta(\vec{y})\cdot\nabla \overline{v}(\vec{y}) - \nabla u_\eta(\vec{y})\cdot(\overline{v}(\vec{y})(\nabla d_\eta)(\vec{y}))\right)dy
        \end{aligned}
    \end{equation} 
    Noting that $\lvert \nabla d_\eta\rvert <C$ for some $C>0$ only depending on $\Omega$ almost everywhere by Rademacher's Theorem, we find that upon normalising, the remainder can be bounded by
    \begin{equation}
        \lVert r_\eta \rVert_{H^{-1}}\leq C\lVert A\rVert_{L^\infty}(1 + \lambda)^{1/2}\lVert d_\eta u_\eta\rVert^{-1},
    \end{equation}
    where $C>0$ depends only on $d_\eta$.
\end{proof}
We now need to freeze the macroscopic variable. This is the content of the following lemma
\begin{lemma}\label{lem: FrozenCoefficients}
    Let $(v_\eta)$ be a sequence of pseudo-eigenvectors obtained from Lemma \ref{lem: PseudoeigenvectorsGeneral}. Then there exists a point $\vec x_0\in \Omega$ and a sequence of pseudo-eigenvectors $\tilde{v}_\eta$ such that
    \begin{itemize}
        \item $\operatorname{supp}(\tilde{v}_\eta)\subset\eta^{-1}\bar{\Omega}$,
        \item $\tilde{v}_\eta\in H^1$ is a weak solution in $\mathbb{R}^n$ of 
         \begin{equation} 
            \left\{
                \begin{aligned}
                    - \mathrm{div} \; \left( A\left( \vec{x}_0, \textbf{P} \vec{y} \right) \nabla \tilde{v}_\eta({\vec y}) \right )& = \lambda \tilde{v}_\eta(\vec{y})+\tilde{r}_\eta \;, \qquad \vec{y} \in \eta^{-1}\bar{\Omega} \\
                    \tilde{v}_\eta(\vec{y})& = 0,\qquad \vec{y}\in\partial\eta^{-1}\Omega,
                \end{aligned}
            \right.
        \end{equation}
        where $\tilde{r}_\eta\in H^{-1}(\mathbb{R}^n)$ is  an error term with norm converging to $0$ as $\eta\to0$.
    \end{itemize}
\end{lemma}
\begin{proof}
    We cover the domain $\bar{\Omega}\eta^{-1}$ by a set of non-overlapping cubes $(Q^\eta_i)_{1\leq i\leq N_\eta}$ of size $[-\frac{\beta_\eta}{2\eta}\,\frac{\beta_\eta}{2\eta}]^n$, where $N_\eta = \mathcal{O}(\beta_\eta^{-n})$. For every $\eta$, denote by $i(\eta)$ the index of the cube where the $L^2$ norm of $v_\eta$ restricted to $Q^\eta_{i(\eta)}$ is maximal. Further, let $\vec{y}^\eta_i$ be the centre of the cube $Q^\eta_i$ and let $\vec{x}^\eta_i = \eta\vec{y}^\eta_i$. We note first that
    \begin{equation}\label{eq: FrozenCoefficients1}
        \lVert v_\eta\mid_{Q^\eta_{i(\eta)}}\rVert > c\beta_\eta^{n/2},
    \end{equation}
    for some $c>0$ small enough. Moreover, since $\vec{x}^\eta_{i(\eta)}\in \Omega$, as $\eta\to0$ there exists a convergent subsequence with some limit $\vec{x}_0\in \bar{\Omega}$. Let $\varphi$ be a smooth bump function satisfying
    \begin{equation}
        \begin{cases}
            0 \leq \varphi \leq 1,\\
            \operatorname{supp}(\varphi)\subset[-1,1]^n,\\
            \varphi=1 \text{ in } [-1/2,1/2]^n.
        \end{cases}
    \end{equation}
    Denote $\varphi_\eta(\cdot): = \varphi\left(\frac{\cdot-\vec{y}^\eta_{i(\eta)}}{\eta^{-1}\beta_\eta}\right)$ and further define
    \begin{equation}
        \tilde{v}_\eta := \frac{\varphi_\eta v_\eta}{\lVert \varphi_\eta v_\eta\rVert}.
    \end{equation}
    It remains to check whether $\tilde{v}_\eta$ makes a good sequence of pseudo-eigenvectors for the fixed-variable operator. First, we bound the error occurred by freezing the coefficients: 
    \begin{equation}\label{eq: FrozenCoefficients2}
        \begin{aligned}
            \langle T_{\vec{x}_0}\tilde{v}_\eta,v \rangle_{H^{-1},H^{1}}&=\langle T_{\eta}\tilde{v}_\eta,v \rangle_{H^{-1},H^{1}}\\
            &+ \underbrace{\lVert \varphi_\eta v_\eta\rVert^{-1}\int_{\eta^{-1}\Omega}\left(A\left( \vec{x}_0, \textbf{P}\vec{y}\right)-A\left( \eta\vec{y}, \textbf{P}\vec{y}\right)\right)\nabla  (\varphi_\eta(\vec{y})v_\eta(\vec{y})) \cdot \nabla \overline{v}(\vec{y})dy}_{A_\eta}.    
        \end{aligned}
    \end{equation}
    Thus, we require an upper bound on $A_\eta$
    \begin{equation}\label{eq: FrozenCoefficients3}
        \begin{aligned}
            A_\eta&:=\lVert \varphi_\eta v_\eta\rVert^{-1}\int_{\eta^{-1}\Omega}\left(A\left( \vec{x}_0, \textbf{P}\vec{y}\right)-A\left( \eta\vec{y}, \textbf{P}\vec{y}\right)\right)\nabla  (\varphi_\eta(\vec{y})v_\eta(\vec{y})) \cdot \nabla \overline{v}(\vec{y})dy\\
            &\leq \lVert \varphi_\eta v_\eta\rVert^{-1}[A]_{\mathcal{C}^{0,\alpha}}(\beta_\eta + \lvert \vec{x}_{i(\eta)}^\eta-\vec{x}_0\rvert)^\alpha(\lVert \nabla \varphi_\eta \rVert_{L^\infty}\lVert v_\eta\rVert + \lVert \varphi_\eta \nabla v_\eta\rVert)\lVert\nabla  v\rVert.
        \end{aligned}
    \end{equation}
    For the first term, we note simply that $\varphi$ can be chosen such that $\lvert \nabla \varphi_\eta\rvert \leq C\eta\beta_\eta^{-1}$. To bound the second term, we make use of Caccioppoli's inequality to obtain 
    \begin{equation}
        \lVert \varphi_\eta \nabla v_\eta \rVert^2 \leq C(1 + \lvert \lambda\rvert) \lVert \varphi_\eta v_\eta\rVert^2 + C\lVert r_\eta\rVert_{H^{-1}}^2.   
    \end{equation}
    Applying these bounds to Equation \eqref{eq: FrozenCoefficients3} and making use of \eqref{eq: FrozenCoefficients1}, we obtain 

    \begin{equation}
        A_\eta \leq C[A]_{\mathcal{C}^{0,\alpha}}(\beta_\eta + \lvert \vec{x}_{i(\eta)}^\eta-\vec{x}_0\rvert)^\alpha(\eta\beta_\eta^{-(n+2)/2} + (1 + \lvert \lambda\rvert)^{1/2}+ \beta_\eta^{-n/2}\lVert r_\eta\rVert_{H^{-1}})\lVert\nabla  v\rVert.
    \end{equation}
    In a next step, we examine $\langle T_{\eta}\tilde{v}_\eta,v \rangle_{H^{-1},H^{1}}$
    \begin{equation}
        \begin{aligned}
            &\langle T_{\eta}\tilde{v}_\eta,v \rangle_{H^{-1},H^{1}}=\lVert \varphi_\eta v_\eta\rVert^{-1}\int_{\eta^{-1}\Omega}A\left( \eta\vec{y}, \textbf{P}\vec{y}\right)\nabla  (\varphi_\eta(\vec{y})v_\eta(\vec{y})) \cdot \nabla \overline{v}(\vec{y})dy\\
            &\quad =\lVert \varphi_\eta v_\eta\rVert^{-1}\langle T_{\eta}v_\eta,
            \varphi_\eta v \rangle_{H^{-1},H^{1}}\\
            & + \underbrace{\lVert \varphi_\eta v_\eta\rVert^{-1}\int_{\eta^{-1}\Omega}A\left( \eta\vec{y}, \textbf{P}\vec{y}\right) \left((\nabla  \varphi_\eta(\vec{y}))v_\eta(\vec{y})) \cdot \nabla \overline{v}(\vec{y}) - \nabla v_\eta \cdot (\nabla \varphi_\eta)(\vec{y})\overline{v}(\vec{y})\right)d\vec{y}}_{B_\eta}.
        \end{aligned}
    \end{equation}
    Since every term in $B_\eta$ includes a factor of $\nabla \varphi_\eta$, and we can again exploit the choice that $\lvert \nabla \varphi_\eta\rvert \leq C\eta\beta_\eta^{-1}$, we have the bound
    \begin{equation}
        B_\eta\leq  C\eta\beta_\eta^{-1}\lVert \varphi_\eta v_\eta\rVert^{-1} \lVert A\rVert_{L^\infty}((1+\lvert \lambda \rvert)^{1/2}\lVert \varphi_\eta v_\eta\rVert + \lVert r_\eta\rVert_{H^{-1}})\lVert v\rVert_{H^1}.
    \end{equation}
    Putting everything together, we have
    \begin{equation}
        \langle T_{\vec{x}_0}\tilde{v}_\eta,v \rangle_{H^{-1},H^{1}} = \lambda \langle \tilde{v}_\eta,v \rangle + \lVert \varphi_\eta v_\eta\rVert^{-1}\langle r_\eta,v\rangle_{H^{-1},H^1}+ A_\eta + B_\eta.
    \end{equation}
    For $\beta_\eta = \max(\eta^{2/(n+3)}, \lVert r_\eta\rVert_{H^{-1}}^{1/n})$ the claim follows.
\end{proof}

Through Lemmas~\ref{lem: PseudoeigenvectorsGeneral} and~\ref{lem: FrozenCoefficients}
we have constructed, for each
$\lambda\in\lim_{\eta\to0}\eta^{2}\sigma\setminus\sigma_{\text{Boundary}}$, a
sequence of compactly supported pseudo-eigenvectors $\tilde v_\eta$ in $H^1(\mathbb{R}^n)$ for the
frozen-coefficient operator $T_{\vec x_0}$ at some $\vec x_0\in\Omega$,
normalised in $L^2$ and with $H^{-1}$-residual tending to $0$ as $\eta\to0$. By
Lemma~\ref{lem: spectral_equivalence}, $\lambda\in\sigma(T_{\vec x_0})\subset
\sigma_{\text{Bloch}}$; hence
$\lim_{\eta\to0}\eta^{-2}\sigma\subset\sigma_{\text{Bloch}}\cup
\sigma_{\text{Boundary}}$. Together with the reverse inclusion from
Proposition~\ref{prop: BlochSpectrum}, this establishes
Theorem~\ref{thm: SpectralLimit} in the critical scaling, modulo the
characterisation of $\sigma_{\text{Boundary}}$ obtained in
Section~\ref{sec:boundarylayer}. An example was shown in Figure~\ref{fig:eigenvalue_critical}, where we observe convergence of the bulk eigenvalues to the Bloch spectrum, interspersed with the boundary spectrum.

\subsection{Non-critical scalings}
\subsubsection{Sub-critical Scaling}
Here, we study the setting
\begin{equation*}
    \lim_{\e\rightarrow 0}a_\e\e^{-1}=0.
\end{equation*}
Let us redefine the map $T_\eta: H^1(a_{\eta}^{-1}\Omega)\rightarrow H^{-1}(a_\eta^{-1}\Omega)$ in light of this scaling as
 \begin{equation}
     u\mapsto \langle T_{\eta}u,v \rangle_{H^{-1},H^{1}}:=\int_{\eta^{-1}\Omega}A\left( a_\eta\vec{y}, \frac{a_\eta}{\eta}\textbf{P}\vec{y}\right)\nabla u(\vec{y}) \cdot \nabla \overline{v}(\vec{y}) d\vec{y}
 \end{equation}
Without loss of generality we may assume $0\in\Omega$. We claim that the sequence of operators $T_\e$ converges strongly to the operator
\begin{equation}\label{eq: SubCriticalLimit}
     T: H^1(\mathbb{R}^n)\rightarrow H^{-1}(\mathbb{R}^n), u\mapsto \langle T_{}u,v \rangle_{H^{-1},H^{1}}:=\int_{\mathbb{R}^n}A\left(\vec{0}, \vec{0}\right)\nabla u(\vec{y}) \cdot \nabla \overline{v}(\vec{y}) d\vec{y}
 \end{equation}
 The intuition behind this was already outlined in \cite{allaire1998bloch}, albeit for the non-rescaled problem. After rescaling, this can be understood as follows: Since $a_\eta \ll \eta$, the domain grows much faster than the number of fundamental cells inside the domain (this would be true when $A$ is periodic in the second variable, but the intuitive idea is the same in the quasiperiodic setting). Therefore, one would expect  the operators coefficients to be increasingly locally constant. This picture aligns well with \eqref{eq: HighFrequencyNew}.

 Let $u \in \mathcal{C}^\infty_c(\R^n)$ and $\e$ be small enough such that $\operatorname{supp}u\subset a_\e^{-1}\O$. We can then also regard $u$ as an element of $H^1_0(a_\e^{-1}\O)$ and hence it makes sense to consider $(T_\eta - T)u$. We obtain
\begin{equation}
    \begin{aligned}
    \lvert\langle (T_\eta - T)u,v\rangle_{H^{-1},H^1}\rvert &= \lvert\int_{\eta^{-1}\Omega}\left(A\left( a_\eta\vec{y}, \frac{a_\eta}{\eta}\textbf{P}\vec{y}\right)-A\left(\vec{0}, \vec{0}\right)\right)\nabla u(\vec{y}) \cdot \nabla \overline{v}(\vec{y}) d\vec{y}\rvert\\
    &\leq [A]_{\mathcal{C}^{0,\alpha}}\left(\frac{a_\eta}{\eta}\right)^\alpha C_{\operatorname{supp}u,\alpha}\lVert u\rVert_{H^1}\lVert v\rVert_{H^1}.
    \end{aligned}
\end{equation}
Since $\mathcal{C}^\infty_c(\R^n)$ is dense in $H^1(\R^n)$, we find that $T_\eta\rightarrow T$ strongly. From this another application of Lemma \ref{lem: spectral_equivalence} immediately yields
\begin{equation}\label{eq: SubCriticalSpectrum}
    \lim_{\eta\rightarrow0}\sigma (T_\eta) \supset\sigma(T).
\end{equation}
 However, since $T_\eta$ is an elliptic operator, we have $\sigma(T_\eta) \subset \R_+$ and since $T$ additionally has constant coefficients, $\sigma (T) = \R_+$. From this it follows that the inclusion in equation \eqref{eq: SubCriticalSpectrum} is actually an equality.
 
\subsubsection{Super-critical Scaling}
In the super-critical setting, we have 
\begin{equation*}
    \lim_{\e\rightarrow 0}a_\e\e^{-1}=\infty.
\end{equation*}
This time, the number of fundamental cells grows faster than the size of the domain. Thus, homogenisation should occur in the rescaled problem. Let
\begin{equation}
    T^h: H^1(\mathbb{R}^n)\rightarrow H^{-1}(\mathbb{R}^n), u\mapsto \langle T^hu,v \rangle_{H^{-1},H^{1}}:=\int_{\eta^{-1}\Omega}A^h\left(\vec{0}\right)\nabla u(\vec{y}) \cdot \nabla \overline{v}(\vec{y}) d\vec{y}
\end{equation}
where $A^h$ is the homogenised operator in the sense of equation \eqref{eq:homogenized_sigma_static}. Denote by $R_\e(z), R^h(z)$ the resolvents of $T_\e -z, T^h-z$. We will show that $R_\e(-1)f\rightarrow R^h(-1)f$ strongly for $f$ in some dense subset of $L^2(\R^n)$.

First, note that since $T^h$ is an elliptic operator with constant coefficients we have $\sigma (T^h) = \mathbb{R}_+$. Similarly, since $T_\eta$ is also elliptic, we have $\sigma(T_\eta) \subset \R_+$. Therefore, $R^h(-1), R_\eta(-1)$ are well defined. Let $\mathcal{F'}$ be a dense subset of $\mathcal{C}^\infty_c$ with respect to the $H^2$ topology and set $\mathcal{F} = \operatorname{Im}((T^h+\operatorname{Id})\mid_{\mathcal{F}'})$. Because $(T^h+\operatorname{Id}):H^2\rightarrow L^2$ is continuous, $\mathcal{F}$ is dense in the image of $T^h+\operatorname{Id}$. Moreover, since $T^h$ has constant coefficients, $R^h(-1)$ maps $L^2$ into $H^{2}$, that is, it is an isomorphism. Therefore, $\mathcal{F}$ is dense in $L^2$. It is thus sufficient to show that $R_{\eta}(-1)f\rightarrow R^h(-1)f$ for all $f \in \mathcal{F}$.

To this end, let $u_\eta \in H^1_0(a_\eta^{-1}\Omega)\subset H^1(\R^n)$ be the unique solution to 
\begin{equation}\label{eq:Supercritical1}
    \begin{cases}
        \begin{aligned}
          - \mathrm{div} \; \left( A\left( a_\eta \vec{y}, \frac{a_\eta}{\eta}\textbf{P} \vec{y} \right) \nabla u_\eta({\vec y}) \right ) + u_\eta& = f(\vec y) \;, \qquad \vec{y} \in a_\eta^{-1}\Omega\\
            u_\eta|_{\partial a_\eta^{-1}\Omega} & = 0.
        \end{aligned}
    \end{cases}
\end{equation}
Due to the ellipticity of $A$, $u_\eta$ is a bounded sequence in $H^1(\R^n)$ which follows from
\begin{equation*}
    c\lVert u_\eta\rVert_{H^1}^2\leq\langle T_\eta u_\eta,u_\eta\rangle_{H^{-1},H^1} + \lVert u_\eta \rVert^2 = \langle f, u_\eta\rangle_{H^{-1},H^1} \leq \lVert f\rVert_{H^{-1}}\lVert u\rVert_{H^1}.
\end{equation*} From Proposition~\ref{prop:grad-split} we thus, upon passing to a subsequence, may assume that there exist $(u_0,\nabla_{\bf P}u_1) \in H^1(\R^n) \times L^2(\R^n, \mathcal{L}_{\textbf{P}}(Y^m))$ such that

\begin{itemize}
    \item $u_\eta\rightharpoonup u_0$ weakly in $L^2(\R^n)$,
    \item $\nabla u_\eta \wdto \nabla u_0 + \nabla_\textbf{P}u_1$ in $L^2$.
\end{itemize}
Moreover, $(u_0,\nabla_\textbf{P}u_1)$ is the unique solution to
\begin{equation}
    \begin{cases}
        -\operatorname{div}_{\textbf{P}}\left(A(\vec 0, \vec y)(\nabla u_0(\vec x) + \nabla_\textbf{P} u_1(\vec x, \vec y))\right)= 0 \quad (\vec x,  \vec y) \in \R^n\times Y^m\\
        -\operatorname{div}\left(A^h(\vec 0)\nabla u_0(\vec x)\right) + u_0(\vec x)= f(\vec x) \quad \vec x \in \R^n.
    \end{cases}
\end{equation}
From the fact that $\mathcal{F}$ is defined as the image of $\mathcal{F}'\subset \mathcal{C}^\infty_c(\mathbb{R}^n)$ under $T^h+\operatorname{Id}$, we have that $u_0\in\mathcal{C}^\infty_c(\mathbb{R}^n)$. Further, since one can identify
\begin{equation}
  \nabla_\textbf{P}  u_1(x,y) =  -\partial_iu_{0}(x) \nabla_\textbf{P}\chi^i(y),
\end{equation}
where $\nabla_\textbf{P}\chi^i(y),$ is defined by the solution of 
\eqref{eq:local_static}, $ \nabla_\textbf{P} u_1$ is actually in $\mathcal{C}^\infty_c(\R^n)\times L^2_\sharp(Y^m)$. 
This implies $ \nabla_\textbf{P}u_1(\vec{x},\frac{a_\eta}{\eta}\textbf{P}\vec{x})\dto  \nabla_\textbf{P}u_1(\vec{x},\vec{y})$.
We thus have

\begin{equation}
    A\left( a_\eta\vec{y}, \frac{a_\eta}{\eta}\textbf{P}\vec{y} \right)\left(\nabla u_0(\vec{y}) + \nabla_\textbf{P}u_1(\vec{y}, \frac{a_\eta}{\eta}\textbf{P}\vec{y})\right)\dto A(\vec 0,\vec{z  })\left(\nabla u_0(\vec{z}) + \nabla_{\mathbf{P}}u_1(\vec{y},\vec {z})\right).
\end{equation}
Using Assumption \ref{ass:A}, we find
\begin{equation}
    \begin{aligned}
        \lVert u_\eta-u_0\rVert^2 &+ c\lVert \nabla u_\eta - \nabla u_0 -\nabla_\textbf{P}u_1(\cdot,\frac{a_\eta}{\eta}\textbf{P}\cdot)\rVert^2\\
        &\leq \int_{\R^n} A\left( a_\eta\vec{y}, \frac{a_\eta}{\eta}\textbf{P}\vec{y} \right)\left(\nabla u_\eta(\vec{y}) -\nabla u_0(\vec{y}) - \nabla_\textbf{P}u_1(\vec{y}, \frac{a_\eta}{\eta}\textbf{P}\vec{y})\right)\cdot \\ 
        &\left(\nabla \overline{u}_\eta(\vec{y}) -\nabla \overline{u}_0(\vec{y}) - \nabla_\textbf{P}\overline{u}_1(\vec{y}, \frac{a_\eta}    {\eta}\textbf{P}\vec{y})\right) + \lvert u_\eta(\vec{y})- u_0(\vec{y})\rvert^2 dx\\
        &=\langle f,u_\eta\rangle + \int A\left( a_\eta\vec{y}, \frac{a_\eta}{\eta}\textbf{P}\vec{y} \right)\left[-\nabla u_\eta(\vec{y})\cdot\left(\nabla \overline{u}_0(\vec{y})
        + \nabla_\textbf{P}\overline{u}_1(\vec{y}, \frac{a_\eta}{\eta}\textbf{P}\vec{y})\right)\right.\\
        &\left.- \left(\nabla {u}_0(\vec{y}) +\nabla_\textbf{P}{u}_1(\vec{y}, \frac{a_\eta}    {\eta}\textbf{P}\vec{y})\right)\cdot \nabla\overline{u}_\eta(\vec{y})\right.\\
        &\left.+ \left(\nabla {u}_0(\vec{y}) + \nabla_\textbf{P}{u}_1(\vec{y}, \frac{a_\eta}    {\eta}\textbf{P}\vec{y})\right)\cdot \left(\nabla \overline{u}_0(\vec{y}) + \nabla_\textbf{P}\overline{u}_1(\vec{y}, \frac{a_\eta}    {\eta}\textbf{P}\vec{y})\right)\right] + \\
        & \lvert u_0(\vec{y})\rvert^2 - u_\eta(\vec{y})\overline{u}_0(\vec{y}) - u_0(\vec{y})\overline{u}_\eta(\vec{y})dx.
    \end{aligned}
\end{equation}
Since every term in the above integral contains at most one $u_\eta$, we can take the two scale cut-and-project limit to find
\begin{equation}
    \begin{aligned}
        \lim_{\eta\rightarrow0} \lVert u_\eta-u_0\rVert^2 &+ c\lVert \nabla u_\eta - \nabla u_0 -\nabla_\textbf{P}u_1(\cdot,\frac{a_\eta}{\eta}\textbf{P}\cdot)\rVert^2\leq \langle f,u_0\rangle -\langle u_0,f\rangle =2i\Im(\langle f,u_0\rangle).
    \end{aligned}
\end{equation}
The claim follows upon taking the real part. Alternatively, one could also note that since $A$ is Hermitian by Assumption \ref{ass:A}, $\langle f,u_0\rangle$ needs to be real-valued, hence the right hand side of the above equation vanishes.

In particular, this implies  $\lVert R_\eta(-1)f - R^h(-1)f\rVert\rightarrow0$ for all $f \in \mathcal{F}$ and hence, $\lim \sigma(R_\eta)\supset\sigma(R^h) = [0,1]$, or alternatively, $\lim _{\eta\rightarrow 0}\sigma(T_\eta)\supset\R_+$.

\section{Boundary Layer Spectrum}\label{sec:boundarylayer}

Throughout this subsection, we assume $\Omega\subset\R^n$ is a bounded domain with $\mathcal{C}^{2}$ boundary. In Figure~\ref{fig:eigenvalue_critical} we showed an example of the spectral convergence predicted by Theorem~\ref{thm: SpectralLimit} in the critical scaling limit. As well as observing convergence of the bulk eigenvalues to the Bloch spectrum, we see that the boundary spectrum $\sigma_{\text{Boundary}}$ is highly sensitive to how the coefficients are truncated at the edge of the domain, which differs for different values of $\eta$. The purpose of this section is to provide a more insightful characterisation of $\sigma_{\text{Boundary}}$. 

We denote the tangent half-space at $\vec{x}_0$ as  $ T_{\vec{x}_0}^+\Omega$ and further set
 $\mathbb{T}_{\vec{x}_0} = \vec{x}_0 + T_{\vec{x}_0}^+\Omega$. This allows us to define the map $T_{\vec{x}_0,\vec{y}_0}: H^1_0(T_{\vec{x}_0}^+\Omega)\rightarrow H^{-1}(T_{\vec{x}_0}^+\Omega)$ such that
\begin{equation}
   u\mapsto \langle T_{\vec{x}_0,\vec{y}_0}u,v \rangle_{H^{-1},H^{1}}:=\int_{T_{\vec{x}_0}^+\Omega}A\left( \vec{x}_0, \vec{y}_0+ \textbf{P}\vec{x}\right)\nabla u(\vec{x}) \cdot \nabla \overline{v}(\vec{x}) dx.
\end{equation}
In this section, we will show that $\sigma_{\text{Boundary}}$ is contained in the spectrum of $T_{\vec{x}_0,\vec{y}_0}$.
\begin{theorem}\label{thm: BoundarySpectrum}
    Define 
    \begin{equation*}
        B=\left\{(\vec{x}_0, \vec{y}_0)\in\partial\Omega\times Y^m\mid\exists\text{ a subsequence $\eta_k$ such that }(\vec{x}_0, \vec{y}_0)=\lim_k(\vec{x}_0,\eta_k^{-1}\textbf{P}\vec{x}_0)\right\},
    \end{equation*}
    then it holds that
    \begin{equation}        \sigma_{\text{Boundary}}\subset\bigcup_{(\vec{x}_0,\vec{y}_0)\in B}\sigma_{}(T_{\vec{x}_0,\vec{y}_0}).
    \end{equation}
\end{theorem}
\begin{remark}
    Without further assumptions on $\eta$, it is known that $\sigma_{\text{Boundary}} = \mathbb{R}_+$ \cite{castro1996remarque}.
\end{remark}
\begin{remark}
    Theorem \ref{thm: BoundarySpectrum} is a refinement of \cite[Theorem 7.11]{allaire1998bloch} to suit the present setting of quasiperiodic coefficients.
\end{remark}
The rest of this article is devoted to proving Theorem~\ref{thm: BoundarySpectrum}. The key observation is that the definition of the $\sigma_{\text{Boundary}}$ forces the associated eigenfunctions to concentrate quite strongly near the boundary. 

For $R>0$, define the \emph{tubular neighbourhood} of $\partial(\Omega)$ by
\begin{equation}
    \mathcal{T}_{R}:=\left\{\vec{x}\in\Omega\mid d(\vec{x}) < R\right\},
\end{equation}
and its complement $\mathcal{T}_{R}^c:=\Omega\setminus\mathcal{T}_{R}$.
Similarly, we also define the tubular neighbourhood of $\partial(\eta^{-1}\Omega)$ by
\begin{equation}
    \mathcal{T}_{R,\eta}:=\left\{\vec{y}\in\eta^{-1}\Omega\mid d_\eta(\vec{y}) < R\right\},
\end{equation}
and its complement $\mathcal{T}_{R,\eta}^c:=\eta^{-1}\Omega\setminus\mathcal{T}_{R,\eta}$.

\begin{lemma}[$L^2$~concentration near the boundary]\label{lem:L2concentration}
    Let $\lambda\in\sigma_{\text{Boundary}}$ and let $(\lambda_\eta,u_\eta)$ be a corresponding sequence of eigenpairs for $T_\eta$ on $\eta^{-1}\Omega$ with $\lVert u_\eta\rVert_{L^2}=1$ and $\sup_\eta\lVert d_\eta u_\eta\rVert\leq C$. Then for every $R>0$,
    \begin{equation}\label{eq:L2tail}
        \lVert u_\eta\rVert^2_{L^2(\mathcal{T}_{R,\eta}^c)}\leq \frac{C^2}{R^2}.
    \end{equation}
    In particular, for any sequence $R_\eta\rightarrow\infty$ we have
    \begin{equation}
        \lVert u_\eta\rVert^2_{L^2(\mathcal{T}_{R_\eta,\eta})}=1 - \mathcal{O}(R_\eta^{-2}).
    \end{equation}
\end{lemma}
\begin{proof}
    On $\mathcal{T}_{R,\eta}^c$ we have $d_\eta\geq R$, then $C>0$ is some constant independent of $R$ such that
    \begin{equation*}
        C^2 \geq \lVert d_\eta u_\eta\rVert^2 \geq \int_{\mathcal{T}_{R,\eta}^c}d_\eta^2\lvert u_\eta\rvert^2\,d\vec{y} \geq R^2\lVert u_\eta\rVert^2_{L^2(\mathcal{T}_{R,\eta}^c)}.
    \end{equation*}
    The second statement follows since $\lVert u_\eta\rVert^2_{L^2(\mathcal{T}_{R_\eta,\eta})} = 1 - \lVert u_\eta\rVert^2_{L^2(\mathcal{T}_{R_\eta,\eta}^c)} \geq 1 - C^2/R_\eta^2 \rightarrow 1$.
\end{proof}

We cover $\partial\Omega$ by $(n-1)$-dimensional cubes $(q^\eta_j)_j$ of side
$\beta_\eta$ (flattening the boundary locally to $\mathbb{R}^{n-1}$, gridding, and
mapping back. The resulting overlap has multiplicity bounded by a constant depending on $\partial\Omega$).
Since $\Omega$ is $\mathcal{C}^2$, the nearest-point projection $\pi:T_\delta\to\partial\Omega$ is
$\mathcal{C}^2$ for some $\delta>0$ \cite{gilbarg1977elliptic}, so for $\beta_\eta\le\delta$ we extend
each cube inwards along $\pi$ to a slab $Q^\eta_j$. The rescaled cubes $\eta^{-1}Q^\eta_j$
then cover $T_{\beta_\eta/\eta,\,\eta}$, with $N_\eta\le C_{\partial\Omega}\,\beta_\eta^{1-n}$.

We now proceed similar to Section \ref{sec:critical}. Let $\mathcal{N}(j,\eta)$ be the set of indices of cubes neighbouring $Q_j^\eta$. Set $\hat{Q}_j^\eta = \bigcup_{j\in\mathcal{N}(j,\eta)}$ and define
\begin{equation}
    \chi^\eta_j(x) = \begin{cases}
        1,\quad x\in Q_j^\eta\\
        0,\quad x\notin \hat{Q}_j^\eta. 
    \end{cases}
\end{equation}
Next, we define
\begin{equation}
    \chi^\eta_r(x) = \begin{cases}
        1,\quad x\in\mathcal{T}_{\beta_\eta/2\eta,\eta}\\
        0,\quad x\notin\mathcal{T}_{2\beta_\eta/3\eta,\eta}.
    \end{cases}
\end{equation}
Furhter, set $\chi_{\eta,j}=\chi_{j(\eta)}\chi^\eta_r.$ We now define another smooth bump function 
\begin{equation}
    \hat{\chi}_{j,\eta}=\begin{cases}
        1, \quad \vec{x}\in\operatorname{supp}\nabla\chi_{j,\eta},\\
        0, \quad \vec{x}\notin\mathcal{T}_{\beta_\eta/\eta,\eta}.
    \end{cases}
\end{equation}
Set now $\Gamma_{j,\eta}=\operatorname{supp}\hat{\chi}_{j,\eta}$
Our aim is to find a sequence of cubes $j(\eta)$ such that $\lVert \chi_\eta u_\eta\rVert$ is lower bounded while at the same time there exists some constant $C_1>0$, independent of $\eta$ such that $\lVert u_\eta\mid_{\Gamma_{j,\eta}}\rVert\leq C_1\lVert \chi_\eta u_\eta\rVert$. To this end, we say an index $j$ is bad, if
this condition is not satisfied. Let $C_2$ be a constant that upper bounds the number of overlapping cubes, which can be chosen independently of $\eta$. We then have
\begin{equation}
   \begin{aligned}
        \sum_{j \text{ bad}}\lVert\chi_\eta u_\eta\rVert^2 &< C_1^{-2}\sum_{j \text{ bad}}\lVert u_\eta\mid_{\Gamma_{j,\eta}}\rVert^2\\
        &\leq 3^{n-1}C_2C_1^{-2}
   \end{aligned}
\end{equation}
since each cube is contained in at most $3^{n-1}$ neighbourhoods. Thus, by choosing $C_1=\sqrt{2\cdot3^{n-1}C_2}$, we find that the mass of bad cubes is bounded:
\begin{equation}
    \sum_{j \text{ bad}} \lVert u_\eta \mid_{Q_j^\eta}\rVert^2 \leq \frac{1}{4}.
\end{equation}
At the same time, the concentration Lemma \ref{lem:L2concentration} yields for $\eta$ small enough
\begin{equation}
    \sum_j\lVert u_\eta \mid_{Q_j^\eta\cap\mathcal{T}_{\beta/2\eta,\eta}}\rVert^2\geq\frac{1}{2}.
\end{equation}
Hence, the mass of the union good cubes intersected with $\mathcal{T}_{\beta_\eta/2\eta,\eta}$ is lower bounded by $1/4$. The pigeonhole principle now implies that there exists some good $j(\eta)$ satisfying
\begin{equation}\label{eq: NiceCube}
    \lVert \chi_\eta u_\eta \rVert\geq C_2\beta_{\eta}^{(n-1)/2}\quad\text{and}\quad\lVert  u_\eta \mid_{\Gamma{j,\eta}}\rVert\leq C_1\lVert u_\eta \rVert.
\end{equation}
where $C_0$ is a constant that bounds the number of overlapping cubes and can be made independent of $\eta$. After having chosen such a good cube, we will drop the subscript $j$ from $\chi_{\eta,j}$ in the following.

\begin{lemma}\label{lem:BoundaryPseudo}
    Let $w_\eta = \frac{\chi_\eta u_\eta}{\lVert \chi_\eta u_\eta\rVert}$. Then it holds that
    \begin{equation}
        \langle T_{\vec{x}_0}w_\eta,v\rangle_{H^{-1},H^1} = \lambda_\eta\langle w_\eta,v\rangle + \langle r_\eta,v\rangle_{H^{-1},H^1},
    \end{equation}
    where $$\lVert r_\eta\rVert\leq C(\eta^{2/3} + \lvert \vec{x}_0 - \vec{x}_{j(\eta)}^\eta\rvert)^\alpha + C\eta^{1/3}$$ and $\operatorname{supp}(r_\eta)\subset \operatorname{supp}(w_\eta)$.
\end{lemma}
\begin{proof}
    First, we will estimate the error due to freezing the macroscopic coefficient
    \begin{align}
        \langle (T_{\vec{x}_0} -T_{\eta})w_\eta,v \rangle_{H^{-1},H^{1}}&=\lVert \chi_\eta u_\eta\rVert^{-1}\int_{\eta^{-1}\Omega}\left(A\left( \vec{x}_0, \textbf{P}\vec{y}\right) - A\left( \eta\vec{y}, \textbf{P}\vec{y}\right)\right)\nabla(\chi_\eta(\vec{y})u_\eta(\vec{y}))\cdot\nabla \overline{v}(\vec{y})dy \nonumber\\
        &\leq \lVert \chi_\eta u_\eta\rVert^{-1}[A]_{\mathcal{C}^{0,\alpha}}(\beta_\eta + \lvert \vec{x}_0 - \vec{x}_{j(\eta)}^\eta\rvert)^\alpha\lVert \chi_\eta u_\eta\rVert_{H^1}\lVert v \rVert_{H^1}. \label{eq:BoundaryPseudo1}
    \end{align}
    Using that $\lvert \nabla \chi_\eta\rvert \leq C_1\beta_\eta^{-1}\eta$,  upon another use of Caccioppoli and elliptic regularity we have
    \begin{equation}\label{eq: Boundary0}
        \lVert \nabla (\chi_\eta u_\eta)\rVert^2 \leq C(1 + \lvert \lambda_\eta\rvert + \eta^2\beta_\eta^{-2})\lVert \chi_\eta u_\eta\rVert^2.
    \end{equation}
    Applying this bound to Equation \eqref{eq:BoundaryPseudo1}, we have
    \begin{equation}
        \lvert \langle (T_{\vec{x}_0} -T_{\eta})w_\eta,v \rangle_{H^{-1},H^{1}}\rvert\leq (1 + \lvert \lambda_\eta\rvert + \eta^2\beta_\eta^{-2})^{1/2}[A]_{\mathcal{C}^{0,\alpha}}(\beta_\eta + \lvert \vec{x}_0 - \vec{x}_{j(\eta)}^\eta\rvert)^\alpha \lVert v \rVert_{H^1}.
    \end{equation}
    Next, we examine $\langle T_{\eta}w_\eta,v \rangle_{H^{-1},H^{1}}$
    \begin{equation}
        \begin{aligned}
             \langle T_{\eta}w_\eta,v \rangle_{H^{-1},H^{1}}&=\int_{\eta^{-1}\Omega}A\left( \eta\vec{y}, \textbf{P}\vec{y}\right)\nabla(\chi_\eta(\vec{y})u_\eta(\vec{y})))\cdot\nabla \overline{v}(\vec{y})dy\\
             &=\lVert \chi_\eta u_\eta\rVert^{-1} \langle T_{\eta}u_\eta,\chi_\eta v \rangle_{H^{-1},H^{1}}\\
              &\quad + \underbrace{\lVert \chi_\eta u_\eta\rVert^{-1}\int_{\eta^{-1}\Omega}A\left( \eta\vec{y}, \textbf{P}\vec{y}\right)((\nabla\chi_\eta)(\vec{y})u_\eta(\vec{y})\cdot\nabla \overline{v}(\vec{y})))}_{A_\eta}\\
              &\quad- \underbrace{\lVert \chi_\eta u_\eta\rVert^{-1}\int_{\eta^{-1}\Omega}A\left( \eta\vec{y}, \textbf{P}\vec{y}\right)\nabla u_\eta(\vec{y})\cdot (\nabla \chi_\eta)(\vec{y})\overline{v}(\vec{y}))dy}_{B_\eta}.
        \end{aligned}
    \end{equation}
    Using the bound on $\nabla \chi_\eta$ as well as the fact that by Equation \ref{eq: NiceCube} $2\lVert \chi_\eta u_\eta\rVert\geq \lVert u_\eta\mid_{\operatorname{supp}\chi_\eta}\rVert$, the term $A_\eta$ can be bounded by 
    \begin{equation}
        A_\eta\leq C\eta\beta_\eta^{-1}\lVert A\rVert_{L^\infty}\lVert v\rVert_{H^1}.
    \end{equation}
   In bounding $B_\eta$, we may make use of Caccioppoli and Equation \ref{eq: NiceCube}                                                     once more to find that 
   \begin{equation}
       \lVert \nabla u_\eta\mid_{\Gamma_{j,\eta}}\rVert^2 \leq C(1 + \lvert \lambda_\eta\rvert + \eta^2\beta_\eta^{-2})\lVert \hat{\chi}_\eta u_\eta\rVert^2\leq C(1 + \lvert \lambda_\eta\rvert + \eta^2\beta_\eta^{-2})\lVert \chi_\eta u_\eta\rVert^2.
   \end{equation}
    where $C$ again only depends on $A,\lambda$ and the dimension $n$. Hence, we find that
    \begin{equation}
        B_\eta \leq C\eta \beta_\eta^{-1}\lVert A\rVert_{L^\infty}\lVert v\rVert_{H^1}
    \end{equation}
    Putting everything together, we obtain
    \begin{equation}
        \begin{aligned}
            \langle T_{\vec{x}_0}w_\eta,v\rangle_{H^{-1},H^1} &= \lVert \chi_\eta u_\eta\rVert^{-1}\langle T_{\eta}u_\eta,\chi_\eta v\rangle_{H^{-1},H^1}+ \langle (T_{\vec{x}_0} -T_{\eta})w_\eta,v \rangle_{H^{-1},H^{1}} + A_\eta + B_\eta\\
            &= \lambda \langle w_\eta,v\rangle + \langle r_\eta,v\rangle_{H^{-1},H^1},
        \end{aligned}
    \end{equation}
    where
    \begin{equation}
        \langle r_\eta,v\rangle_{H^{-1},H^1} = \langle (T_{\vec{x}_0} -T_{\eta})w_\eta,v \rangle_{H^{-1},H^{1}} + A_\eta + B_\eta.
    \end{equation}
    At a latter point, we will see that we need $\beta_\eta^2\ll\eta$. Thus, by choosing $\beta_\eta = \eta^{2/3}$, we find that $\lVert r_\eta\rVert\leq C(\eta^{2/3} + \lvert \vec{x}_0 - \vec{x}_{j(\eta)}^\eta\rvert)^\alpha + C\eta^{1/3}$.
\end{proof}
In the next step towards proving Theorem~\ref{thm: BoundarySpectrum}, we have to estimate the error caused by straightening out the boundary. To this end, note that since $\Omega$ is $\mathcal{C}^{2}$, we can locally write $\partial \Omega$ as a graph \cite{evans2022partial}. That is, there exists open neighbourhoods $U,V$ of $x_0$ and a map $\gamma: \mathbb{R}^{n-1}\rightarrow\mathbb{R}\in \mathcal{C}^2(\mathbb{R}^{n-1})$ such that the following hold
\begin{itemize}
    \item The map $\Phi: \mathbb{R}^n\rightarrow\mathbb{R}^n, \vec{x}\rightarrow \vec{x} + \gamma(\textbf{P}_{\vec{x}_0}\vec{x})\vec{n}_{x_0}$ is an isomorphism from $U\cap\Omega$ onto $V\cap \mathbb{T}_{\vec{x}_0}^+(\Omega)$ (here, $\textbf{P}_{\vec{x}_0}$ denotes the projection onto the tangent space of $\vec{x}_0$),
    \item $\Phi(\vec{x}_0) = \vec{x}_0$ and $D\Phi(\vec{x}_0)=\operatorname{Id}$,
    \item $\Phi^{-1}(\vec{y}) = \vec{y} - \gamma(\textbf{P}_{\vec{x}_0}\vec{y})\vec{n}_{\vec{x}_0}$,
    \item $\det D\Phi = \det D\Phi^{-1} = 1$.
\end{itemize}
If we let
\begin{equation}\Phi_\eta:\eta^{-1}U\rightarrow\mathbb{R}^n,\vec{y}\mapsto\eta^{-1}\Phi(\eta\vec{y}) - \eta^{-1}\vec{x}_0,
\end{equation}
then, upon passing to a subsequence, we may assume that $\eta^{-1}\textbf{P}\vec{x}_0$ converges to some point $\vec{y}_0\in Y^m$. Let $v\in\mathcal{C}^\infty_0(T_{\vec{x}_0}^+\Omega)$, then for $\eta$ small enough we have $v\circ\Phi_\eta\in\mathcal{C}_0^\infty(\eta^{-1}\Omega)$ and hence
\begin{align}
        &\langle T_{\vec{x}_0}w_\eta, v\circ\Phi_\eta\rangle_{H^{-1},H^1} = \int_{\eta^{-1}\Omega}A\left(\vec{x}_0, \textbf{P}\vec{y}\right)\nabla w_\eta(\vec{y)}\cdot\nabla(v\circ\Phi_\eta(\vec{y}))dy\nonumber\\
        &=\int_{T_{\vec{x}_0}^+\Omega}A\left(\vec{x}_0, \textbf{P}\Phi^{-1}_\eta(\vec{y})\right)(\nabla w_\eta)\circ\Phi^{-1}_\eta(\vec{y})\cdot D\Phi_\eta(\vec{y})\nabla v(\vec{y})dy \nonumber\\
        &=\langle T_{\vec{x}_0,\vec{y}_0}w_\eta\circ\Phi_\eta^{-1},v\rangle_{H^{-1},H^1}\nonumber\\
        &+ \underbrace{\int_{T_{\vec{x}_0}^+\Omega}\left(A\left(\vec{x}_0, \textbf{P}\Phi^{-1}_\eta(\vec{y})\right) - A\left(\vec{x}_0, \vec{y}_0+\textbf{P}\vec{y}\right)\right)(\nabla w_\eta)\circ\Phi^{-1}_\eta(\vec{y})\cdot D\Phi_\eta(\vec{y})\nabla v(\vec{y})dy}_{A_\eta}\nonumber\\
        &+\underbrace{\int_{T_{\vec{x}_0}^+\Omega}A\left(\vec{x}_0, \textbf{P}\Phi^{-1}_\eta(\vec{y})\right)\left((\nabla w_\eta)\circ\Phi^{-1}_\eta(\vec{y}) - \nabla(w_\eta\circ\Phi_\eta^{-1}(\vec{y})\right)\cdot D\Phi_\eta(\vec{y})\nabla v(\vec{y})dy}_{B_\eta}\nonumber\\
        &+\underbrace{\int_{T_{\vec{x}_0}^+\Omega}A\left(\vec{x}_0, \textbf{P}\Phi^{-1}_\eta(\vec{y})\right)(\nabla w_\eta)\circ\Phi^{-1}_\eta(\vec{y})\cdot (D\Phi_\eta(\vec{y}) - \operatorname{Id})\nabla v(\vec{y})dy}_{C_\eta}. \label{eq:Boundary1}
\end{align}
For $\eta$ small enough, the support of $w_\eta$ is contained in $\eta^{-1}U$, so the above is well defined. We thus have to take care of $A_\eta,B_\eta$ and $C_\eta$. To this end, we notice that by our definition of $\beta$, $\lvert D\Phi_\eta - \operatorname{Id}\rvert < C\eta^{2/3}$, since $\operatorname{supp}(w_\eta)\subset B_{R\eta^{-1/3}}(\vec{x_0})$. Hence,
\begin{equation}
    B_\eta ,C_\eta\leq C\lVert A\rVert_\infty\lVert \nabla v\rVert \lVert \nabla w_\eta\circ\Phi_\eta^{-1}\rVert\eta^{2/3}.
\end{equation}
Note further that by testing $T_\eta w_\eta$ against $w_\eta$, one finds that there exist some $C>0$, depending only on $\gamma$ from Assumption \ref{ass:A} such that for $\eta$  small enough,
\begin{equation}
    \lVert \nabla w_\eta\circ\Phi_\eta^{-1}\rVert^2\leq C(\lambda + 1)\lVert w_\eta\rVert =C(\lambda + 1),
\end{equation}
thus yielding
\begin{equation}
    B_\eta ,C_\eta\leq C(1+\lambda)\lVert A\rVert_\infty\lVert \nabla v\rVert \eta^{2/3}.
\end{equation}
To take care of $A_\eta$, we make use of Hölder continuity and expand $\Phi_\eta^{-1}$ to obtain
\begin{equation}
    \begin{aligned}
        \lvert A\left(\vec{x}_0, \vec{y}_0  +\textbf{P}\vec{y}\right) - &A\left(\vec{x}_0, \textbf{P}\Phi_\eta^{-1}(\vec{y})\right)\lvert\leq [A]_{\mathcal{C}^{0,\alpha}}\lvert\vec{y}_0 + \textbf{P}\vec{y} - \textbf{P}\Phi_\eta^{-1}(\vec{y})\rvert^\alpha\\
    \end{aligned}
\end{equation}
where the norm $\lvert \cdot \rvert$ is to be understood as a norm on $Y^m$.

Since $\eta^{-1}\textbf{P}\vec{x}_0\rightarrow\vec{y}_0$ and using the fact that $\operatorname{supp}(w_\eta)\subset B_{R\eta^{-1/3}}(\eta^{-1}\vec{x}_0)$, we see that $A_\eta$ tends to zero, since
\begin{equation}\label{eq:Boundary2}
    \Phi_\eta^{-1}(y) = y + \eta^{-1}\vec{x}_0 + \mathcal{O}(\eta\lvert \textbf{P}_{\vec{x}_0}\vec{y}\rvert^2).
\end{equation}
Therefore, one can bound $A_\eta$ as follows
\begin{equation}
    A_\eta \leq C(1+\lambda)[A]_{\mathcal{C}^{0,\alpha}} (\lvert \vec{y}_0-\eta^{-1}\textbf{P}\vec{x}_0\rvert + C\eta^{1/3})^\alpha\lVert \nabla v\rVert 
\end{equation}
We thus obtain
\begin{equation}
    \langle T_{\vec{x}_0,\vec{y}_0} w_\eta\circ\Phi_\eta^{-1}, v\rangle_{H^{-1},H^1}=\lambda\langle w_\eta\circ\Phi_\eta^{-1}, v\rangle_{H^{-1},H^1} +\langle \tilde{r}_\eta, v\rangle_{H^{-1},H^1},
\end{equation}
where $\lim_{\eta\rightarrow0}\lVert\tilde{r}_\eta\rVert_{H^{-1}}=0 $. The constructed pseudo-eigenfunction thus satisfies the assumption of Lemma \ref{lem: spectral_equivalence}, yielding
\begin{equation}
    \sigma_{\text{Boundary}} \subset \bigcup_{(\vec{x}_0,\vec{y}_0)\in B}\sigma(T_{\vec{x}_0,\vec{y}_0}).
\end{equation}

Unfortunately, as in \cite{allaire1998bloch}, we are unable to show a stricter statement for the boundary spectrum. However, we do want to point out that the following result from \cite{allaire1998bloch}, still holds:
\begin{theorem}\label{thm: SpectralLimit2}
    The limit spectrum in the Hausdorff sense satisfies
    \begin{equation}        \lim\eta^{2}\sigma_\eta= \sigma_{\text{Bloch}}\cup\left(\bigcup_{(\vec{x}_0,\vec{y}_0)\in B}\sigma(T_{\vec{x}_0,\vec{y}_0})\right).
    \end{equation}
\end{theorem}
\begin{proof}
    One inclusion directly follows from Theorem \ref{thm: BoundarySpectrum} and \ref{thm: SpectralLimit}. 
    
    The proof of the converse inclusion is similar to that of Proposition \ref{prop: BlochSpectrum}, we therefore only sketch it: Let $(\vec{x}_0,\vec{y}_0)\in B$ and $\lambda \in \sigma(T_{\vec{x}_0,\vec{y}_0})$. If also $\lambda \in \sigma(T_{\vec{x}_0})$, then immediately $\lambda \in \lim_{\eta\rightarrow0}\eta^2\sigma_\eta$. Otherwise, there must exist a sequence of boundary localised pseudo-eigenfunctions (in the sense of $\sigma_{\text{Boundary}}$), for else we could again construct pseudo-eigenfunctions for $T_{\vec{x_0}}$ as in the Lemmas \ref{lem: PseudoeigenvectorsGeneral} and \ref{lem: FrozenCoefficients} . After shifting and multiplying with a smooth bump function as in Proposition \ref{prop: BlochSpectrum}, one can proceed with running the estimates of Lemma \ref{lem:BoundaryPseudo} and Equation \eqref{eq:Boundary1} in reverse to prove that this sequence also makes for good pseudo-eigenfunctions for the operator $T_\eta$.
\end{proof}

\begin{remark}
    The proofs of Theorem \ref{thm: BoundarySpectrum} and \ref{thm: SpectralLimit2} extend readily to polytopes. Since each face is flat, the cubes $Q_j^\eta$ can be taken as standard Euclidean cubes aligned with the faces, and the boundary-straightening error is identically zero. If the limit point $\vec{x}_0$ lies in the interior of a face, the construction is identical to the above. If $\vec{x}_0$ lies on an edge or at a corner, the cutoff $\chi_{j(\eta)}^\eta$ is taken to equal $1$ on all cubes touching the edge or corner, with the transition to $0$ occurring over the adjacent layer of cubes; the gradient bound $\lvert\nabla\chi_{j(\eta)}^\eta\rvert \leq C\eta\beta_\eta^{-1}$ remains valid. The limiting operator is then the corresponding wedge or cone operator $T_{\vec{x}_0, C(\vec{x}_0)}$, where $C(\vec{x}_0)$ denotes the tangent cone at $\vec{x}_0$.
\end{remark}

\section{Conclusion}\label{sec:conclusion}
In this paper, we have studied the spectral asymptotics of elliptic operators with quasiperiodic coefficients generated by the cut-and-project method. Using two-scale convergence adapted to cut-and-project structures, we have established convergence of the spectrum in the low-frequency homogenisation regime (Section~2). In the high-frequency regime (Section~3), we have introduced a rescaling approach that transforms the problem to an expanding domain, and we have shown that the rescaled spectrum converges to $\sigma_{\text{Bloch}}\cup\sigma_{\text{Boundary}}$ in the critical scaling, and to $\mathbb{R}_+$ in both the sub-critical and super-critical scalings. In Section~4, we have shown that, for domains with $\mathcal{C}^2$ boundary, the boundary layer spectrum is contained in $\bigcup_{(\vec{x}_0,\vec{y}_0)\in B}\sigma(T_{\vec{x}_0,\vec{y}_0})$, the union of spectra of half-space operators with frozen macroscopic variable over all attained boundary points $\vec{x}_0,\vec{y}_0$.

\section*{Acknowledgements}
C. Thalhammer was supported in part by the Swiss National Science Foundation grant number 200021-236472. C. Thalhammer thanks Habib Ammari for providing support and the freedom to pursue this topic independently.

\section*{Code Availability}
The code used to generate Figure \ref{fig:eigenvalue_critical} is openly available at \url{https://github.com/cthalhammer/HighFrequencyAsymptotics}.

\printbibliography

\end{document}